\documentclass[12pt]{amsart}
\usepackage{amssymb}
\usepackage{mathtools}
\usepackage{txfonts}
\usepackage{eucal}
\usepackage{extarrows}

\usepackage[all]{xy}

\usepackage{tikz}

\usepackage{graphicx}
\usepackage{float}

\usepackage{xspace}

\usepackage[a4paper,body={16.3cm,22.8cm},centering]{geometry}
\usepackage[colorlinks,final,backref=page,hyperindex,pdftex]{hyperref}

\newcommand{\nc}{\newcommand}
\newcommand{\delete}[1]{}

\nc{\mlabel}[1]{\label{#1}}  
\nc{\mcite}[1]{\cite{#1}}  
\nc{\mref}[1]{\ref{#1}}  
\nc{\mbibitem}[1]{\bibitem{#1}} 

\delete{
\nc{\mlabel}[1]{\label{#1}  
{\hfill \hspace{1cm}{\small\tt{{\ }\hfill(#1)}}}}
\nc{\mcite}[1]{\cite{#1}{\small{\tt{{\ }(#1)}}}}  
\nc{\mref}[1]{\ref{#1}{{\tt{{\ }(#1)}}}}  
\nc{\mbibitem}[1]{\bibitem[\bf #1]{#1}} 
}

\newtheorem{theorem}{Theorem}[section]
\newtheorem{prop}[theorem]{Proposition}
\newtheorem{lemma}[theorem]{Lemma}
\newtheorem{coro}[theorem]{Corollary}

\theoremstyle{definition}
\newtheorem{defn}[theorem]{Definition}
\newtheorem{remark}[theorem]{Remark}
\newtheorem{exam}[theorem]{Example}
\newtheorem{prop-def}{Proposition-Definition}[section]

\newcommand\cal[1]{\mathcal{#1}}

\newcommand\alphlist{a,b,c,d,e,f,g,h,i,j,k,l,m,n,o,p,q,r,s,t,u,v,w,x,y,z}
\newcommand\Alphlist{A,B,C,D,E,F,G,H,I,J,K,L,M,N,O,P,Q,R,S,T,U,V,W,X,Y,Z}
\newcommand\getcmds[3]{\expandafter\newcommand\csname #2#1\endcsname{#3{#1}}}
\makeatletter
\@for\x:=\alphlist\do{\expandafter\getcmds\expandafter{\x}{frak}{\mathfrak}}
\@for\x:=\Alphlist\do{\expandafter\getcmds\expandafter{\x}{frak}{\mathfrak}}
\makeatother

\nc{\bfk}{{\bf k}}
\font\cyr=wncyr10

\newfont{\scyr}{wncyr10 scaled 550}
\nc{\sha}{\mbox{\cyr X}}
\nc{\ssha}{\mbox{\bf \scyr X}}

\nc{\id}{\mathrm{id}}
\nc{\Id}{\mathrm{Id}}
\nc{\lbar}[1]{\overline{#1}}
\nc{\ot}{\otimes}
\nc{\dep}{\mathrm{dep}}
\nc{\ver}{\mathrm{ver}}

\nc{\tred}[1]{\textcolor{red}{#1}} \nc{\tgreen}[1]{\textcolor{green}{#1}}
\nc{\tblue}[1]{\textcolor{blue}{#1}} \nc{\tpurple}[1]{\textcolor{purple}{#1}}
\nc{\tcyan}[1]{\textcolor{cyan}{#1}} 
\nc{\tblk}[1]{\textcolor{black}{#1}}

\nc{\li}[1]{\tpurple{\underline{Li:}#1 }}
\nc{\liadd}[1]{\tpurple{#1}}
\nc{\xing}[1]{\tblue{\underline{Xing:}#1 }}
\nc{\yuan}[1]{\tred{\underline{Yuan:}#1 }}
\nc{\markus}[1]{\tred{\underline{Markus:} #1}}
\nc{\dominique}[1]{\tpurple{\underline{Dominique: }#1 }}
\long\def\ignore#1{}

\usetikzlibrary{calc,shapes.geometric}
\tikzset{
baseon/.style={baseline={($(#1)+(0,-0.58ex)$)}},
baseon/.default=current bounding box.center,
every picture/.style=baseon,
lst/.style={},
dst/.style={fill,circle,inner sep=1.5pt,outer sep=0pt},
ddst/.style={diamond,draw,inner sep=1pt},
eest/.style={ellipse,draw,inner sep=1pt,minimum size=2ex},
}
\newcommand\treeo[2][]{\tikz[x=0.7cm,y=0.7cm,line width=0.15ex,
every node/.style={font=\scriptsize,inner sep=1pt,label distance=1pt},#1]{%
\coordinate (o) at (0,0);#2}}%
\newcommand\treeoo[2][]{\treeo[#1]{\path (o) node[dst,name=o]{};#2}}%
\def\zzz#1`#2...#3`#4...#5`#6@{%
--++(#1)
node[dst,label={#5:$#6$},name=#2]{}
node[midway,auto,#3]{$#4$}
}
\def\ddd#1`#2`#3@{+(#1)node[ddst,name=#2]{$#3$}}
\def\eee#1`#2`#3@{+(#1)node[eest,name=#2]{$#3$}}
\def\xxx#1`#2@{node[midway,auto,inner sep=1pt,#1]{$#2$}}
\def\pp#1`#2`#3@{node[dst,label={#2:$#3$},pos=#1]{}}
\def\oo#1`#2`#3@{\path (o) node[dst,label={#2:$#3$},name=o,#1]{};}
\def\eoo#1`#2@{\node[eest,name=o,#1] at (o) {$#2$};}

\nc{\dnx}{\Delta_n A} \nc{\dx}{\Delta A} \nc{\dgp}{{\rm deg_{P}}}
\nc{\dgt}{{\rm deg_{T}}} \nc{\dg}{{\rm deg}} \nc{\ida}{ID($A$)} \nc{\tu}{\tilde{u}} \nc{\tv}{\tilde{v}}
\nc{\nr}{\calr_n} \nc{\nz}{\calz_n} \nc{\fun}{\cala_{n,d}}
 \nc{\fbase}{\calb} \nc{\LF}{\mathrm{RF}} \nc{\FFA}{\mathrm{LF}} \nc{\irr}{\mathrm{Irr}}
 \nc{\result}{\bfk\mathrm{Irr}(S_n)}  \nc{\I}{I_{\mathrm{ID},n}^0}
 \nc{\nrs}{\calr_n^\star} \nc{\ii}{\mathrm{I}} \nc{\iii}{\mathrm{II}}
\nc{\intl}{{\rm int}}\nc{\ws}[1]{{#1}}\nc{\deleted}[1]{\delete{#1}}\nc{\plas}{placements\xspace}

\nc{\bim}[1]{#1}  \nc{\shaop}{\sha_{\Omega}^{+}}  \nc{\shao}{\sha_{\Omega}}
\nc{\bbim}[2]{#1 #2} \nc{\bbbim}[2]{#1,\, #2} \nc{\RBF}{{\rm RBF}}
\nc{\frb}{F_{\RB}} \nc{\shaf}{\ssha_{\tiny{\Omega}}} \nc{\sham}{\diamond_{\tiny{\Omega}}}
\nc{\lf}{\lfloor} \nc{\rf}{\rfloor} \nc{\shan}{\ssha_{\lambda}}
\nc{\rlex}{{\rm {lex}}} \nc{\bb}{\Box} \nc{\ra}{\rightarrow}
\nc{\e}{{\rm {e}}}
\nc{\DDF}{\mathrm{DD}(X,\,\Omega)}\nc{\DTF}{\mathrm{DT}(X,\,\Omega)} \nc{\DT}{\mathrm{DT}'(\Omega,\,V)}
\nc{\bra}{\mathrm{bra}} \nc{\bre}{\mathrm{bre}}
\nc{\dec}{\mathrm{dec}} \nc{\diamondw}{\diamond_{w}}
\nc{\type}{\mathrm{type}}

\nc\caF[1]{\cal{F}_{#1}(X,\,\Omega)}
\nc\calt{\cal{T}(X,\,\Omega)} \nc\caltn{\cal{T}_n(X,\,\Omega)}
\nc\calta{\cal{T}_0(X,\,\Omega)}
\nc\caltb{\cal{T}_1(X,\,\Omega)}
\nc\caltc{\cal{T}_2(X,\,\Omega)}
\nc\caltd{\cal{T}_3(X,\,\Omega)}
\nc\caltm{\cal{T}_m(X,\,\Omega)}
\nc\caltx{\cal{T}(X)}
\nc\calf{\cal{F}(X,\,\Omega)}
\nc\fram{\frak{M}(\Omega,\, X)}
\nc\shaw{\sha^{NC}_w(\Omega,\, X)}
\nc\dw{\diamond_w} \nc\dl{\diamond_\ell}
\nc\shal{\sha^{NC}_\ell(X,\, \Omega)} \nc\shav{\sha^{NC}_w(\Omega,\, V)} \nc\shat{\sha^{NC,1}_w(\Omega,\, T^{+}(V))}
\nc{\cfo}{\cal{F}(X,\,\Omega)}

\nc{\lar}{\varinjlim}
\nc\XO{(X,\,\Omega)}
\def\cxo#1#2;{\cal{#1}#2\XO}
\nc\lrf[2]{B_{#2}^+(#1)}
\nc{\fd}{\mathrm{\text{typed angularly decorated planar rooted trees}}}
\nc{\rb}{\mathrm{RBFWs}} \nc{\dfw}{\mathrm{DFW{(X)}}} \nc{\tfw}{\mathrm{TFW{(X)}}}
\nc{\tfv}{\mathrm{TFW{(V)}}}

\def\Ve#1,#2,#3;{\vee_{#1,\,(#2,\,#3)}}
\def\bigv#1;#2;#3;{\bigvee\nolimits_{#1}^{#2;\,#3}}
\nc\rjt[2]{\mathrel{\mathop{\longrightarrow}\limits^{#1\hfill}_{\hfill#2}}}
\nc{\pl}{\cal{PLF}}
\nc{\tr}{\cal{RTF}}
\nc{\im}{\mathrm{Im}}
\nc{\ff}{\cal{F}_\Omega}
\nc{\tm}{T_\Omega}
\nc{\calp}{\cal{P}}
\makeatletter
\nc\dd{\@ifnextchar'{\ddA}{\ddB}}
\def\ddA'#1;{\rhd'_{#1\,}}
\def\ddB#1;{\rhd_{#1\,}}
\makeatother
\begin{document}

\title[Free pre-Lie family algebras]{Free pre-Lie family algebras}
%

\author{Dominique Manchon}
\address{Laboratoire de Math\'ematiques Blaise Pascal,
CNRS--Universit\'e Clermont-Auvergne,
3 place Vasar\'ely, CS 60026,
F63178 Aubi\`ere, France}
\email{Dominique.Manchon@uca.fr}

\author{Yuanyuan Zhang} \address{School of Mathematics and Statistics,
Lanzhou University, Lanzhou, 730000, P. R. China}
\email{zhangyy17@lzu.edu.cn}

%

\date{\today}

\begin{abstract}
In this paper, we first define the pre-Lie family algebra associated to a dendriform family algebra in the case of a commutative semigroup. Then we construct a pre-Lie family algebra via typed decorated rooted trees, and we prove the freeness of this pre-Lie family algebra. We also construct pre-Lie family operad in terms of typed labeled rooted trees, and we obtain that the operad of pre-Lie family algebras is isomorphic to the operad of typed labeled rooted trees, which generalizes the result of F. Chapoton and M. Livernet. In the end, we construct Zinbiel and pre-Poisson family algebras and generalize results of M. Aguiar.
\end{abstract}

\subjclass[2010]{
16W99, 
16S10, 
08B20, 
}

\keywords{}

\maketitle

\tableofcontents

\setcounter{section}{0}

\allowdisplaybreaks
\section{Introduction}

Rota-Baxter family algebras were proposed by L. Guo in 2009 \cite{Guo09}, as a generalization of Rota-Baxter algebras. This was the first example of ``family algebraic structure", where an algebraic structure interacts with a semigroup in a way that remains to be understood in full generality. In this paper, we propose a notion of pre-Lie family algebra compatible with the notion of dendriform family algebra introduced in \cite{ZG}, describe the corresponding free objects and operad in terms of decorated (resp. labeled) typed rooted trees. Interestingly enough, the semigroup at play must be commutative in this case. In the last section, we propose a family counterpart of Zinbiel algebras and Aguiar's pre-Poisson algebras \cite{A}.\\

It is well known that
 the first direct link between Rota-Baxter algebras and dendriform algebras was given by Aguiar~\cite{A} who showed that a Rota-Baxter algebra of weight $0$ carries a dendriform algebra structure.
In ~\cite{ZG}, the authors gave the definition of dendriform family algebras and also studied the similar property between Rota-Baxter family algebras and dendriform family algebras. Free Rota-Baxter family algebras were studied in~\cite{ZG,ZGM2} and free dendriform family algebras were studied in ~\cite{ZGM}. L. Foissy recently proposed a unified framework encompassing both dendriform family algebras indexed by a semigroup and matching dendriform algebras indexed by a set \cite{ZGG20}, by means of the new concept of extended diassociative semigroup \cite{Foi20}.\\

Rooted trees are a useful tool for several areas of Mathematics, such as in the study of vector fields~\cite{Cay}, numerical analysis~\cite{Br} and quantum field theory~\cite{CK}.
The work of the British mathematician Arthur Cayley in the 1850s can now be considered as the prehistory of pre-Lie algebras. F. Chapoton and M. Livernet~\cite{CL} first showed that the free pre-Lie algebra generated by a set $X$ is given by grafting of $X$-decorated rooted trees (see also \cite{DL} for an approach which does not use operads). Pre-Lie structures on rooted trees lead to the Butcher-Connes-Kreimer Hopf algebra of rooted forests, which attracted many scholars to study, such as~\cite{CK, GL, Hof, Mur}. The basic setting for Butcher series is provided by this combinatorial Hopf algebra, which has become an indispensable tool in the analysis of numerical integration~\cite{HLW}. Lie-Butcher series come with a Hopf algebra of planar rooted forests closely related to post-Lie algebras~\cite{MW}.\\

In Section \ref{sec:rp}, we define the pre-Lie family algebra associated to a dendriform family algebra in the case of a commutative semigroup. L. Foissy~\cite{Foi18} considered typed decorated trees to describe free multiple pre-Lie algebras. The decoration is a map from the set of vertices into a set $X$, whereas the type is a map from the set of edges into another set $\Omega$. In his approach, no semigroup structure is required on the set $\Omega$. In Section \ref{sec:pre}, we show that  the free pre-Lie family algebra is also given by decorated rooted trees, typed by some commutative semigroup $\Omega$, whose multiplication plays an essential role.\\

Operads were first named and rigorously defined by J. Peter May in his 1972 book~\cite{May72}, which investigated the applications of operads to loop spaces and homotopy analysis.
There has been a renewed interest in this theory after the breakthrough brought along by the introduction of the Koszul duality for operads by V. Ginzburg and M. Kapranov~\cite{GK94}. That renaissance \cite{L96} gave rise to research in many areas, from algebraic topology to theoretical physics~\cite{MSS}, which have continued to yield important results to our days.
F. Chapoton and M. Livernet defined the underlying operad of pre-Lie algebras in terms of rooted trees, which sheds light on the relationship between them. In Section \ref{sec:oper}, we prove that the operad of pre-Lie family algebras is isomorphic to the operad of typed decorated rooted trees, which
generalizes the result of F. Chapoton and M. Livernet~\cite{CL}.\\

Finally, we give a brief account of some other family algebraic structures which may be of some interest: we introduce Zinbiel family and pre-Poisson family algebras in Section~\mref{sec:pp}, and describe the link between them, thus generalizing some results of M. Aguiar \cite{A} (Proposition~\mref{prop:poi} and Proposition~\mref{prop:zp}).\\

\noindent {\bf Notation.}
 Throughout this paper, let $\bfk$ be a unitary commutative ring which will be the base ring of all modules, algebras, as well as linear maps.
 Algebras are unitary associative algebras but not necessary commutative. 
 \section{From dendriform family algebras to pre-Lie family algebras}
 \mlabel{sec:rp}
 \subsection{Reminder on Rota-Baxter and dendriform family algebras}
The first family algebra structures which appeared in the literature are Rota-Baxter family algebras proposed by Li Guo in 2009 \cite{Guo09}. They arise naturally in renormalization of quantum field theory~\cite[Proposition~9.1]{FBP}, see also~\cite{DK}.
\begin{defn}~\rm{(\mcite{FBP,Guo09})}
Let $\Omega$ be a semigroup and $\lambda\in \bfk$ be given.
A {\bf Rota-Baxter family} of weight $\lambda$ on an algebra $R$ is a collection of linear operators $(P_\omega)_{\omega\in\Omega}$ on $R$ such that
\begin{equation*}
P_{\alpha}(a)P_{\beta}(b)=P_{\alpha\beta}\bigl( P_{\alpha}(a)b  + a P_{\beta}(b) + \lambda ab \bigr),\, \text{ for }\, a, b \in R\,\text{ and }\, \alpha,\, \beta \in \Omega.
\end{equation*}
The pair $(R,\, (P_\omega)_{\omega \in \Omega})$ is called a {\bf Rota-Baxter family algebra} of weight $\lambda$.
\mlabel{def:pp}
\end{defn}
A well-known example of Rota-Baxter family algebra of weight $-1$, with $\Omega=(\mathbb Z,+)$, is given by the algebra of Laurent series $R=\bfk[z^{-1},z]]$, where the operator $P_\omega$ is the projection onto the subspace $R_{<\omega}$ generated by $\{z^k,\,k<\omega\}$ parallel to the supplementary subspace $R_{\ge \omega}$ generated by $\{z^k,\,k\ge\omega\}$.\\

The concept of dendriform family algebras was proposed in~\mcite{ZG}, as a generalization of dendriform algebras invented by Loday~\mcite{Lod93} in the study of algebraic $K$-theory.

\begin{defn}\mlabel{defn:dend}\cite{ZG}
Let $\Omega$ be a semigroup. A {\bf dendriform family algebra} is a $\bfk$-module $D$ with a family of binary operations $(\prec_\omega,\succ_\omega)_{\omega\in\Omega}$ such that for $ x, y, z\in D$ and $\alpha,\beta\in \Omega$,
\begin{align}
(x\prec_{\alpha} y) \prec_{\beta} z=\ & x \prec_{\bim{\alpha \beta}} (y\prec_{\beta} z+y \succ_{\alpha} z), \mlabel{eq:ddf1} \\
(x\succ_{\alpha} y)\prec_{\beta} z=\ & x\succ_{\alpha} (y\prec_{\beta} z),\mlabel{eq:ddf2} \\
(x\prec_{\beta} y+x\succ_{\alpha} y) \succ_{\bim{\alpha \beta}}z  =\ & x\succ_{\alpha}(y \succ_{\beta} z). \mlabel{eq:ddf3}
\end{align}
\end{defn}
Any Rota-Baxter family algebra has an underlying dendriform family algebra structure, and even two distinct ones if the weight $\lambda$ is different from zero \mcite{ZG,ZGM}.
\subsection{Pre-Lie family algebras}
 Pre-Lie algebras were introduced in 1963 independently by Gerstenhaber and Vinberg~\mcite{Ger,Vin}, see~\cite{CL} for more references and examples, including an explicit description of the free pre-Lie algebra on a vector space. Now we propose the concept of left pre-Lie family algebras.

 \begin{defn}
 Let $\Omega$ be a commutative semigroup.
 A left pre-Lie family algebra is a vector space $A$ together with binary operations
 $\rhd_\omega:A\times A\ra A$ for $\omega\in\Omega$ such that
 \begin{equation}
 x\rhd_\alpha(y\rhd_\beta z)-(x\rhd_\alpha y)\rhd_{\alpha\beta}z=y\rhd_\beta(x\rhd_\alpha z)-(y\rhd_\beta x)\rhd_{\beta\alpha}z,
 \mlabel{eq:pre}
 \end{equation}
 where $x,y,z\in A$ and $\alpha,\beta\in\Omega.$
 \end{defn}

 \begin{theorem}
Let $\Omega$ be a commutative semigroup and let $(A,(\prec_\omega,\succ_\omega)_{\omega\in\Omega})$ be a dendriform family algebra. Define binary operations
$$x\rhd_\omega y:=x\succ_\omega y-y\prec_\omega x,\text{ for }\, \omega\in\Omega.$$
Then $(A, (\rhd_\omega)_{\omega\in\Omega})$ is a pre-Lie family algebra.
\mlabel{thm:dp}
\end{theorem}

\begin{proof}
On the left hand side, we have
\begin{align*}
&x\rhd_\alpha(y\rhd_\beta z)-(x\rhd_\alpha y)\rhd_{\alpha\beta}z\\
={}&x\rhd_\alpha(y\succ_\beta z-z\prec_\beta y)-(x\succ_\alpha y-y\prec_\alpha x)\rhd_{\alpha\beta}z\\
={}&x\succ_\alpha(y\succ_\beta z-z\prec_\beta y)-(y\succ_\beta z-z\prec_\beta y)\prec_\alpha x\\
&-\Big((x\succ_\alpha y-y\prec_\alpha x)\succ_{\alpha\beta}z-z\prec_{\alpha\beta}(x\succ_\alpha y-y\prec_\alpha x)\Big)\\
={}&x\succ_\alpha(y\succ_\beta z)-x\succ_\alpha(z\prec_\beta y)-(y\succ_\beta z)\prec_\alpha x+(z\prec_\beta y)\prec_\alpha x\\
& -(x\succ_\alpha y)\succ_{\alpha\beta}z+(y\prec_\alpha x)\succ_{\alpha\beta}z+z\prec_{\alpha\beta}(x\succ_\alpha y)-z\prec_{\alpha\beta}(y\prec_\alpha x)\\
={}&(x\prec_\beta y)\succ_{\alpha\beta}z+(x\succ_\alpha y)\succ_{\alpha\beta}z-x\succ_\alpha(z\prec_\beta y)-(y\succ_\beta z)\prec_\alpha x\\
& +z\prec_{\beta\alpha}(y\prec_\alpha x)+z\prec_{\beta\alpha}(y\succ_\beta x)-(x\succ_\alpha y)\succ_{\alpha\beta}z\\
& +(y\prec_\alpha x)\succ_{\alpha\beta}z+z\prec_{\alpha\beta}(x\succ_\alpha y)-z\prec_{\alpha\beta}(y\prec_\alpha x)\\
&\hspace{3cm}\text{(by Eqs.~(\mref{eq:ddf1}-\mref{eq:ddf3}))}\\
={}&(x\prec_\beta y)\succ_{\alpha\beta}z-x\succ_\alpha(z\prec_\beta y)-(y\succ_\beta z)\prec_\alpha x\\
& +z\prec_{\beta\alpha}(y\succ_\beta x)
 +(y\prec_\alpha x)\succ_{\alpha\beta}z+z\prec_{\alpha\beta}(x\succ_\alpha y).
\end{align*}
On the right hand side, we have
\begin{align*}
&y\rhd_\beta(x\rhd_\alpha z)-(y\rhd_\beta x)\rhd_{\beta\alpha}z\\
={}&y\rhd_\beta(x\succ_\alpha z-z\prec_\alpha x)-(y\succ_\beta x-x\prec_\beta y)\rhd_{\beta\alpha}z\\
={}&y\succ_\beta(x\succ_\alpha z-z\prec_\alpha x)-(x\succ_\alpha z-z\prec_\alpha x)\prec_\beta y\\
& -\Big((y\succ_\beta x-x\prec_\beta y)\succ_{\beta\alpha}z-z\prec_{\beta\alpha}(y\succ_\beta x-x\prec_\beta y)\Big)\\
={}&y\succ_\beta(x\succ_\alpha z)-y\succ_\beta(z\prec_\alpha x)-(x\succ_\alpha z)\prec_\beta y+(z\prec_\alpha x)\prec_\beta y\\
& -(y\succ_\beta x)\succ_{\beta\alpha}z+(x\prec_\beta y)\succ_{\beta\alpha}z+z\prec_{\beta\alpha}(y\succ_\beta x)-z\prec_{\beta\alpha}(x\prec_\beta y)\\
={}&(y\succ_\beta x)\succ_{\beta\alpha}z+(y\prec_\alpha x)\succ_{\beta\alpha}z-y\succ_\beta(z\prec_\alpha x)-(x\succ_\alpha z)\prec_\beta y\\
& +z\prec_{\alpha\beta}(x\prec_\beta y)+z\prec_{\alpha\beta}(x\succ_\alpha y)-(y\succ_\beta x)\succ_{\beta\alpha}z\\
&\hspace{3cm}\text{(by Eqs.~(\mref{eq:ddf1}-\mref{eq:ddf3}))}\\
& +(x\prec_\beta y)\succ_{\beta\alpha}z+z\prec_{\beta\alpha}(y\succ_\beta x)-z\prec_{\beta\alpha}(x\prec_\beta y)\\
={}&(y\prec_\alpha x)\succ_{\beta\alpha}z-y\succ_\beta(z\prec_\alpha x)-(x\succ_\alpha z)\prec_\beta y\\
& +z\prec_{\alpha\beta}(x\succ_\alpha y)
+(x\prec_\beta y)\succ_{\beta\alpha}z+z\prec_{\beta\alpha}(y\succ_\beta x).
\end{align*}
Now we see, using the commutativity of the semigroup $\Omega$, that the $i$-th term in the expansion of the left hand side equals to the $\sigma(i)$-th term in the
 expansion of the right hand side, where $\sigma$ is the following permutation of order $6$:
 \begin{equation*}
\begin{pmatrix}
     i \\
     \sigma(i)
\end{pmatrix}
=
\begin{pmatrix}
1 & 2 & 3 & 4 & 5 & 6 \\
5 & 3 & 2 & 6 & 1 & 4
\end{pmatrix}.
\end{equation*}
Thus
$$x\rhd_\alpha(y\rhd_\beta z)-(x\rhd_\alpha y)\rhd_{\alpha\beta}z=y\rhd_\beta(x\rhd_\alpha z)-(y\rhd_\beta x)\rhd_{\beta\alpha}z$$
holds. This completes the proof.
\end{proof}

\section{Free pre-Lie family algebras}
\mlabel{sec:pre}
\subsection{The construction of pre-Lie family algebras}\mlabel{sub:dd}
In this subsection, we apply typed decorated rooted trees
to construct free pre-Lie family algebras.
For this, let us first recall typed decorated rooted trees studied in~\cite{BHZ,Foi18}.\\

\noindent For a rooted tree $T$, denote by $V(T)$ (resp. $E(T)$) the set of its vertices (resp. edges).

\begin{defn}\cite{BHZ}
Let $X$ and $\Omega$ be two sets. An {\bf $X$-decorated $\Omega$-typed (abbreviated typed decorated) rooted tree}
is a triple $T = (T,\dec, \type)$, where
\begin{enumerate}
\item $T$ is a rooted tree.
\item $\dec: V(T)\ra X$ is a map.
\item $\type: E(T)\ra \Omega$ is a map.
\end{enumerate}
\mlabel{defn:tdtree}
\end{defn}
In other words, vertices of $T$ are decorated by elements of $X$ and edges of $T$ are decorated by elements of $\Omega$.\\

\begin{remark}
Let $X$ and $\Omega$ be two sets. The following trees are the same.
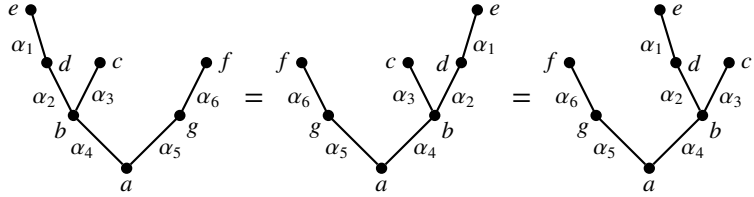
\[
\treeo{
\oo`-90`a@
\draw (o)
\zzz1,1`r...swap`\alpha_5...[inner sep=0pt]-45`g@
\zzz0.5,1`rr...swap`\alpha_6...0`f@
;
\draw (o)
\zzz-1,1`l...`\alpha_4...[inner sep=0pt]-135`b@
\zzz-0.5,1`ll...`\alpha_2...0`d@
\zzz-0.3,1`lll...`\alpha_1...180`e@
;
\draw (l)
\zzz0.5,1`lr...swap`\alpha_3...0`c@
;
}
=
\treeo{
\oo`-90`a@
\draw (o)
\zzz-1,1`l...`\alpha_5...[inner sep=0pt]-135`g@
\zzz-0.5,1`ll...`\alpha_6...180`f@
;
\draw (o)
\zzz1,1`r...swap`\alpha_4...[inner sep=0pt]-45`b@
\zzz0.5,1`rr...swap`\alpha_2...180`d@
\zzz0.3,1`rrr...swap`\alpha_1...0`e@
;
\draw (r)
\zzz-0.5,1`rl...`\alpha_3...180`c@
;
}
=
\treeo{
\oo`-90`a@
\draw (o)
\zzz1,1`r...swap`\alpha_4...[inner sep=0pt]-45`b@
\zzz0.5,1`rr...swap`\alpha_3...0`c@
;
\draw (o)
\zzz-1,1`l...`\alpha_5...[inner sep=0pt]-135`g@
\zzz-0.5,1`ll...`\alpha_6...180`f@
;
\draw (r)
\zzz-0.5,1`rl...`\alpha_2...0`d@
\zzz-0.3,1`rll...`\alpha_1...0`e@
;
}
\]
\end{remark}

For $n\geq 0$, let $\caltn$ denote the set of typed decorated rooted trees with $n$ vertices.
Denote by
$$\calt:= \bigsqcup_{n\geq 0}\caltn\,\text{ and }\, \bfk\calt :=\bigoplus_{n\geq 0}\bfk\caltn.$$
The degree $|T|$ of a typed decorated rooted tree is by definition its number of vertices.
\begin{defn}
Let $X$ be a set and let $\Omega$ be a commutative semigroup. For $S,T\in\calt$ and $\omega\in\Omega$, define
\begin{equation*}
S\rhd_\omega T:=\sum_{v\in \ver(T)} S\rjt{\omega}{v}T.
\end{equation*}
The notation $S\rjt{\omega}{v}T$ means that
\begin{itemize}
\item the tree $S$ is grafted on $T$ at vertex $v$ by means of adding a new edge typed by $\omega$ between the root of $S$ and $v$,
\item each edge below the vertex $v$ has its type multiplied by $\omega$,
\item the other edges keep their types unchanged.
\end{itemize}
\mlabel{defn:pre}
\end{defn}

\begin{exam}
Let $\Omega$ be a commutative semigroup. Let
$$S=\treeo{
\oo`-90`f@
\draw (o)
\zzz-1,1`l...`\beta_1...90`g@
;
\draw (o)
\zzz1,1`r...swap`\beta_2...[inner sep=0pt]-45`j@
\zzz0.5,1`rr...swap`\beta_4...90`i@
;
\draw (r)
\zzz-0.5,1`rl...`\beta_3...90`h@
;
}\,\text{ and }\,
T=\treeo{
\oo`-90`a@
\draw (o)
\zzz-1,1`l...pos=0.3`\alpha_1...[inner sep=0pt]-135`\tred{v}@
\zzz-0.5,1`ll...`\alpha_2...90`c@
;
\draw (o)
\zzz0,1`a...`...90`d@
;
\draw (o)
\zzz1,1`r...swap`\alpha_3...90`e@
;
}.$$
Then we have
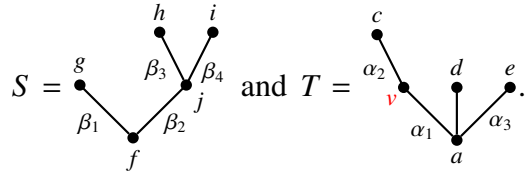
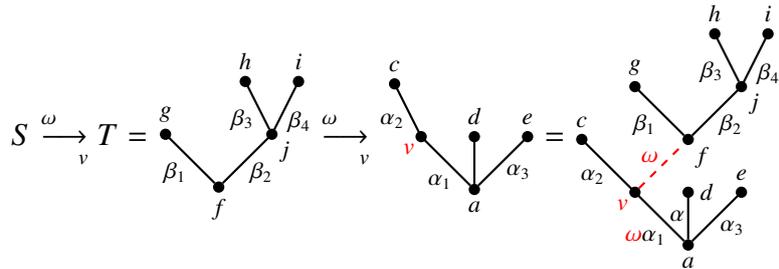
$$S\rjt{\omega}{v}T=\treeo{
\oo`-90`f@
\draw (o)
\zzz-1,1`l...`\beta_1...90`g@
;
\draw (o)
\zzz1,1`r...swap`\beta_2...[inner sep=0pt]-45`j@
\zzz0.5,1`rr...swap`\beta_4...90`i@
;
\draw (r)
\zzz-0.5,1`rl...`\beta_3...90`h@
;
}\rjt{\omega}{v} \treeo{
\oo`-90`a@
\draw (o)
\zzz-1,1`l...pos=0.3`\alpha_1...[inner sep=0pt]-135`\tred{v}@
\zzz-0.5,1`ll...`\alpha_2...90`c@
;
\draw (o)
\zzz0,1`a...`...90`d@
;
\draw (o)
\zzz1,1`r...swap`\alpha_3...90`e@
;
}=\treeo{
\oo`-90`a@
\draw (o)
\zzz-1,1`l...pos=0.3`\tred{\omega}\alpha_1...[inner sep=0pt]-135`\tred{v}@
\zzz-1,1`ll...`\alpha_2...90`c@
;
\draw (o)
\zzz0,1`a...`\alpha...0`d@
;
\draw (o)
\zzz1,1`r...swap`\alpha_3...90`e@
;
\path (l)++(1,1)
node[dst,label={[inner sep=0pt]-45:$f$},name=o2]{}
;
\draw[red, dashed] (l)--(o2)
node[midway,auto]{$\omega$}
;
\draw (o2)
\zzz-1,1`o2l...`\beta_1...90`g@
;
\draw (o2)
\zzz1,1`o2r...swap`\beta_2...[inner sep=0pt]-45`j@
\zzz0.5,1`o2rr...swap`\beta_4...90`i@
;
\draw (o2r)
\zzz-0.5,1`o2rl...`\beta_3...90`h@
;
}.$$
\end{exam}

\begin{lemma}
Let $\Omega$ be a commutative semigroup. For $v,v'\in U$ and $\alpha,\beta\in \Omega$, we have
\begin{equation*}
S\rjt{\alpha}{v}(T\rjt{\beta}{v'}U)=T\rjt{\beta}{v'}(S\rjt{\alpha}{v}U).
\mlabel{eq:one}
\end{equation*}
\mlabel{lemma:one}
\end{lemma}

\begin{proof}
For $\alpha,\beta\in\Omega$ and $v,v'\in U,$ we have
\begin{align*}
S\rjt{\alpha}{v}(T\rjt{\beta}{v'}U)&=\treeo{
\eoo`U@
\draw (o)
\eee-1,1`l`T@
(o)--(l)\xxx`\beta@
\pp0`90`v'@
;
\draw (o)
\eee1,1`r`S@
(o)--(r)\xxx swap`\alpha@
\pp0`90`v@
;
}=\treeo{
\eoo`U@
\draw (o)
\eee-1,1`l`S@
(o)--(l)\xxx`\alpha@
\pp0`90`v@
;
\draw (o)
\eee1,1`r`T@
(o)--(r)\xxx swap`\beta@
\pp0`90`v'@
;
}\\
&=T\rjt{\beta}{v'}(S\rjt{\alpha}{v}U).
\end{align*}
 Both terms are obtained by grafting $S$ on $U$ at vertex $v$ and grafting $T$ on $U$ at vertex $v'$. They are equivalent because the result does not depend on the order in which both operations are performed. Indeed,
\begin{itemize}
\item the new edge which is below $S$ is of type $\alpha$ and the new edge which is below $T$ is of type $\beta;$
\item edges which are below both vertices $v$ and $v'$ have their types multiplied  by $\alpha\beta (=\beta\alpha);$
\item edges which are below $v$ and not below $v'$ have their types multiplied  by $\alpha;$
\item edges which are below $v'$ and not below $v$ have their types multiplied  by $\beta;$
\item the other edges keep their types unchanged.
\end{itemize}
\end{proof}

\noindent For better understanding Lemma~\mref{lemma:one}, we give the following example.
\begin{exam}
Let $\Omega$ be a commutative semigroup. Let $U=\treeo{
\oo`-90`f@
\draw (o)
\zzz-1,1`l...`\beta_1...90`g@
;
\draw (o)
\zzz1,1`r...swap`\beta_2...[inner sep=0pt]-45`\tred{v'}@
\zzz0.5,1`rr...swap`\beta_4...0`\tred{v}@
;
\draw (r)
\zzz-0.5,1`rl...`\beta_3...90`h@
;
\draw (rr)
\zzz-0.5,1`rrl...swap`\beta_5...90`a@
;
}$. Then we have
$$S\rjt{\alpha}{v}(T\rjt{\beta}{v'}U)=
\treeo{
\oo`-90`f@
\draw (o)
\zzz-1,1`l...`\beta_1...90`g@
;
\draw (o)
\zzz1,1`r...swap`\tred{\beta\alpha}\beta_2 (=\tred{\alpha\beta}\beta_2)...[inner sep=0pt]-45`\tred{v'}@
\zzz1,1`rr...swap`\tred{\alpha}\beta_4...[inner sep=0pt]-45`\tred{v}@
\zzz-0.5,1`rrl...`\beta_5...90`a@
;
\draw (r)
\zzz-1,1`rl...`\beta_3...90`h@
;
\draw[red] (rr)
\eee0.5,1`rrr`S@
;
\draw[red,dashed](rr)--(rrr)
\xxx swap`\alpha@
;
\draw[red](r)
\eee-0.5,2`ra`T@
;
\draw[red,dashed](r)--(ra)
\xxx pos=0.8`\beta@
;
}=T\rjt{\beta}{v'}(S\rjt{\alpha}{v}U).$$
\end{exam}

\begin{lemma}\mlabel{lemma:two}
Let $\Omega$ be a commutative semigroup. For $v\in T,v'\in U$ and $\alpha,\beta\in \Omega$, we have
\begin{equation}
S\rjt{\alpha}{v}(T\rjt{\beta}{v'}U)=(S\rjt{\alpha}{v}T)\rjt{\alpha\beta}{v'}U.
\mlabel{eq:two}
\end{equation}
\end{lemma}

\begin{proof}
For $\alpha,\beta\in\Omega$ and $v\in T,v'\in U,$ the equality of both terms can be summarized as follows:
\begin{align*}
S\rjt{\alpha}{v}(T\rjt{\beta}{v'}U)&=\treeo{
\eoo`U@
\draw (o)
\eee0,1.25`a`T@
(o)--(a)\xxx`\alpha\beta@
\pp0`[inner sep=-1pt]45`v'@
;
\draw (a)
\eee0,1.25`aa`S@
(a)--(aa)\xxx`\alpha@
\pp0`[inner sep=-1pt]45`v@
;
}
=(S\rjt{\alpha}{v}T)\rjt{\alpha\beta}{v'}U.
\end{align*}
Indeed, both terms have the same underlying decorated rooted tree. To precise the type of each edge of the left hand side of Eq.~(\mref{eq:two}), we can proceed in two steps:
\begin{itemize}
\item the first step consists in grafting $T$ on $U$ at vertex $v'$, and the new edge is of type $\beta$. The edges which are below the vertex $v'$ of $U$ have their types multiplied by  $\beta;$
\item the second step consists in grafting $S$ on the result $T\rjt{\beta}{v'}U$ at vertex $v$ of $T$, the new edge is of type $\alpha$, and the edges (including the new edge produced by the first step) which are below the vertex $v$ have their types multiplied by  $\alpha.$ So combining the first step and the second step, the new edge between $T$ and $U$ should have its type multiplied by  $\alpha\beta$, and the edges which are below the vertex $v'$ of $U$ have their types multiplied by $\alpha\beta$ totally.
\end{itemize}
Similarly for the right-hand side:
\begin{itemize}
\item the first step consists in grafting $S$ on $T$ at vertex $v$, and the new edge is of type $\alpha$. The edges which are below the vertex $v$ of $T$ have their types multiplied by  $\alpha;$
\item the second step consists in grafting $S\rjt{\alpha}{v}T$ on the result at vertex $v'$ of $U$, the new edge is of type $\alpha\beta$, and the edges which are below the vertex $v'$ of $U$ have their types multiplied by  $\alpha\beta.$
\end{itemize}
So both sides of Eq.~(\mref{eq:two}) coincide.
\end{proof}

\begin{exam}
Let $\Omega$ be a commutative semigroup. Let
$$U=\treeo{
\oo`-90`a_1@
\draw (o)
\zzz-1,1`l...`\omega_2...180`a_2@
\zzz0.5,1`lr...`\omega_1...90`a_5@
;
\draw (o)
\zzz1,1`r...swap`\omega_3...[inner sep=0pt]-45`a_3@
\zzz0.5,1`rr...swap`\omega_5...0`\tcyan{v'}@
\zzz-0.5,1`rrl...swap`\omega_6...90`a_6@
;
\draw (r)
\zzz-0.5,1`rl...`\omega_4...90`a_4@
;
}\,\text{ and }\, T=\treeo{
\oo`-90`b_1@
\draw (o)
\zzz-1,1`l...`\delta_2...180`\tred{v}@
\zzz0.5,1`lr...`\delta_3...90`b_4@
;
\draw (o)
\zzz1,1`r...swap`\delta_1...0`b_3@
;
}.$$
 Then
\begin{align*}
S\rjt{\alpha}{v}(T\rjt{\beta}{v'}U)&=
\treeo{
\oo`-90`a_1@
\draw (o)
\zzz-1,1`l...`\omega_2...180`a_2@
\zzz0.5,1`lr...`\omega_1...90`a_5@
;
\draw (o)
\zzz1,1`r...swap`\tred\alpha\tcyan\beta\omega_3...[inner sep=0pt]-45`a_3@
\zzz1,1`rr...swap`\tred\alpha\tcyan\beta\omega_5...[inner sep=0pt]-45`\tcyan{v'}@
\zzz-0.5,1`rrl...`\omega_6...90`a_6@
;
\draw (r)
\zzz-0.5,1`rl...`\omega_4...90`a_4@
;
\draw[dashed,cyan] (rr)
\zzz1,1`oo,black...swap`\tred\alpha\beta...[inner sep=0pt]-45`\tblk{b_1}@
;
\draw (oo)
\zzz-0.5,1`ool...`\tred\alpha\delta_2...0`\tred{v}@
\zzz0.5,1`oolr...swap,pos=0.8`\delta_3...90`b_4@
;
\draw (oo)
\zzz1,1`oor...swap`\delta_1...0`b_3@
;
\draw (ool)
\eee-0.5,1`ooll`S@
;
\draw[red,dashed](ool)--(ooll)
\xxx`\alpha@
;
}
=\treeo{
\oo`-90`a_1@
\draw (o)
\zzz-1,1`l...`\omega_2...180`a_2@
\zzz0.5,1`lr...`\omega_1...90`a_5@
;
\draw (o)
\zzz1,1`r...swap`\tcyan{\alpha\beta}\omega_3...[inner sep=0pt]-45`a_3@
\zzz1,1`rr...swap`\tcyan{\alpha\beta}\omega_5...[inner sep=0pt]-45`\tcyan{v'}@
\zzz-0.5,1`rrl...`\omega_6...90`a_6@
;
\draw (r)
\zzz-0.5,1`rl...`\omega_4...90`a_4@
;
\draw[dashed,cyan] (rr)
\zzz1,1`oo,black...swap`\alpha\beta...[inner sep=0pt]-45`\tblk{b_1}@
;
\draw (oo)
\zzz-0.5,1`ool...`\tred\alpha\delta_2...0`\tred{v}@
\zzz0.5,1`oolr...swap,pos=0.8`\delta_3...90`b_4@
;
\draw (oo)
\zzz1,1`oor...swap`\delta_1...0`b_3@
;
\draw (ool)
\eee-0.5,1`ooll`S@
;
\draw[red,dashed](ool)--(ooll)
\xxx`\alpha@
;
}\\
&=(S\rjt{\alpha}{v}T)\rjt{\alpha\beta}{v'}U.
\end{align*}
\end{exam}

\begin{prop}
Let $X$ be a set and let $\Omega$ be a commutative semigroup. Then $\bigl((\bfk\calt,(\rhd_\omega)_{\omega\in\Omega}\bigr)$ is a pre-Lie family algebra.
\mlabel{prop:comm}
\end{prop}
\begin{proof}
For $S,T,U\in\calt$ and $\alpha,\beta\in\Omega$, we have
\begin{align*}
& S\rhd_\alpha(T\rhd_\beta U)-(S\rhd_\alpha T)\rhd_{\alpha\beta}U\\
={}&\sum_{v\in\ver(T)\sqcup\ver(U)}\sum_{v'\in\ver(U)}S\rjt{\alpha}{v}(T
\rjt{\beta}{v'}U)-\sum_{v\in\ver(T)}\sum_{v'\in\ver(U)}(S
\rjt{\alpha}{v}T)\rjt{\alpha\beta}{v'}U\\
={}&\sum_{v\in\ver(T)}\sum_{v'\in\ver(U)}S\rjt{\alpha}{v}(T
\rjt{\beta}{v'}U)
+\sum_{v\in\ver(U)}\sum_{v'\in\ver(U)}S\rjt{\alpha}{v}(T
\rjt{\beta}{v'}U)\\
&-\sum_{v\in\ver(T)}\sum_{v'\in\ver(U)}(S
\rjt{\alpha}{v}T)\rjt{\alpha\beta}{v'}U\\
={}&\sum_{v\in\ver(U)}\sum_{v'\in\ver(U)}S\rjt{\alpha}{v}(T
\rjt{\beta}{v'}U)\quad(\text{by Lemma~\mref{lemma:two}})\\
={}&\sum_{v\in\ver(U)}\sum_{v'\in\ver(U)}T\rjt{\beta}{v'}(S
\rjt{\alpha}{v}U)\quad(\text{by Lemma~\mref{lemma:one}})\\
={}&T\rhd_\beta(S\rhd_\alpha U)-(T\rhd_\beta S)\rhd_{\beta\alpha}U.
\end{align*}
\end{proof}

\noindent For better understanding Proposition~\mref{prop:comm}, we give the following example.
\begin{exam}
Let $\Omega$ be a commutative semigroup.
Let $T=\bullet_b$ and $U=\treeo{
\oo`-90`c_1@
\draw (o)
\zzz0,1`r...swap`\gamma...90`c_2@
;}$. Then
\begin{align*}
& S\rhd_\alpha(T\rhd_\beta U)-(S\rhd_\alpha T)\rhd_{\alpha\beta}U\\
={}&\sum_{v\in\ver(T)\sqcup\ver(U)}\sum_{v'\in\ver(U)}S\rjt{\alpha}{v}(T
\rjt{\beta}{v'}U)-\sum_{v\in\ver(T)}\sum_{v'\in\ver(U)}(s
\rjt{\alpha}{v}T)\rjt{\alpha\beta}{v'}U\\
={}&\sum_{v\in\ver(T)}\sum_{v'\in\ver(U)}S\rjt{\alpha}{v}(T
\rjt{\beta}{v'}U)
+\sum_{v\in\ver(U)}\sum_{v'\in\ver(U)}s\rjt{\alpha}{v}(T
\rjt{\beta}{v'}U)
-\sum_{v\in\ver(T)}\sum_{v'\in\ver(U)}(S
\rjt{\alpha}{v}T)\rjt{\alpha\beta}{v'}U\\
={}&\treeo{
\oo`-90`c_1@
\draw (o)
\zzz-1,1`l...`\gamma...90`c_2@
;
\draw[dashed,red] (o)
\zzz1,1`r...swap`\tcyan\alpha\beta...[inner sep=-1.5pt]0`b@
node[circle,draw,solid,minimum size=2.7ex]{}
;
\draw[cyan] (r)
\eee0,1.2`ra`S@
;
\draw[dashed,cyan] (r)--(ra)
\xxx swap,pos=0.8`\alpha@
;
}+\treeo{
\oo`-90`c_1@
\draw (o)
\zzz0,1`a...swap`\tcyan\alpha\tred\beta\gamma...0`c_2@
;
\draw[dashed,red] (a)
\zzz0,1.2`aa...swap`\tcyan\alpha\tred\beta...[inner sep=-1.5pt]0`b@
node[circle,draw,solid,minimum size=2.7ex]{}
;
\draw[cyan] (aa)
\eee0,1.2`aaa`S@
;
\draw[dashed,cyan] (aa)--(aaa)
\xxx swap,pos=0.8`\tcyan\alpha@
;
}+\treeo{
\oo`-90`c_1@
\draw (o)
\zzz0,1`a...pos=0.7`\gamma...90`c_2@
;
\draw[dashed,red] (o)
\zzz1,1`r...swap`\tred\beta...[inner sep=-1.5pt]0`b@
node[circle,draw,solid,minimum size=2.7ex]{}
;
\draw[cyan] (o)
\eee-1,1`l`S@
;
\draw[dashed,cyan] (o)--(l)
\xxx`\tcyan\alpha@
;
}+\treeo{
\oo`-90`c_1@
\draw (o)
\zzz0,1`a...`\tred\beta\gamma...180`c_2@
;
\draw[dashed,red] (a)
\zzz0,1.2`aa...`\tred\beta...[inner sep=-1.5pt]0`b@
node[circle,draw,solid,minimum size=2.7ex]{}
;
\draw[cyan] (o)
\eee1,1`r`S@
;
\draw[dashed,cyan] (o)--(r)
\xxx swap`\tcyan\alpha@
;
}
+\treeo{
\oo`-90`c_1@
\draw (o)
\zzz0,1`a...`\tcyan\alpha\gamma...180`c_2@
;
\draw[cyan] (a)
\eee0,1.2`aa`S@
;
\draw[dashed,cyan] (a)--(aa)
\xxx`\tcyan\alpha@
;
\draw[dashed,red] (o)
\zzz1,1`r...swap`\tred\beta...[inner sep=-1.5pt]0`b@
node[circle,draw,solid,minimum size=2.7ex]{}
;
}+\treeo{
\oo`-90`c_1@
\draw (o)
\zzz0,1`a...`\tred\beta\tcyan\alpha\gamma...90`c_2@
;
\draw[dashed,red] (a)
\zzz-1,1`al...`\tred\beta...[inner sep=-1.5pt]180`b@
node[circle,draw,solid,minimum size=2.7ex]{}
;
\draw[cyan] (a)
\eee1,1`ar`S@
;
\draw[dashed,cyan] (a)--(ar)
\xxx swap`\tcyan\alpha@
;
}
-\left(
\treeo{
\oo`-90`c_1@
\draw (o)
\zzz0,1`a...swap`\tcyan{\alpha\beta}\gamma...0`c_2@
;
\draw[dashed,cyan] (a)
\zzz0,1.2`aa...swap`\alpha\beta...[inner sep=-1.5pt]0`b@
node[circle,draw,solid,minimum size=2.7ex]{}
;
\draw[red] (aa)
\eee0,1.2`aaa`S@
;
\draw[dashed,red] (aa)--(aaa)
\xxx swap,pos=0.8`\tred\alpha@
;
}+\treeo{
\oo`-90`c_1@
\draw (o)
\zzz1,1`r...swap`\gamma...90`c_2@
;
\draw[dashed,cyan] (o)
\zzz-1,1`l...`\alpha\beta...[inner sep=-1.5pt]180`b@
node[circle,draw,solid,minimum size=2.7ex]{}
;
\draw[red] (l)
\eee0,1.2`la`S@
;
\draw[dashed,red] (l)--(la)
\xxx pos=0.8`\alpha@
;
}
\right)\\
={}&\treeo{
\oo`-90`c_1@
\draw (o)
\zzz0,1`a...pos=0.7`\gamma...90`c_2@
;
\draw[dashed,red] (o)
\zzz1,1`r...swap`\tred\beta...[inner sep=-1.5pt]0`b@
node[circle,draw,solid,minimum size=2.7ex]{}
;
\draw[cyan] (o)
\eee-1,1`l`S@
;
\draw[dashed,cyan] (o)--(l)
\xxx`\tcyan\alpha@
;
}+\treeo{
\oo`-90`c_1@
\draw (o)
\zzz0,1`a...`\tred\beta\gamma...180`c_2@
;
\draw[dashed,red] (a)
\zzz0,1.2`aa...`\tred\beta...[inner sep=-1.5pt]0`b@
node[circle,draw,solid,minimum size=2.7ex]{}
;
\draw[cyan] (o)
\eee1,1`r`S@
;
\draw[dashed,cyan] (o)--(r)
\xxx swap`\tcyan\alpha@
;
}
+\treeo{
\oo`-90`c_1@
\draw (o)
\zzz0,1`a...`\tcyan\alpha\gamma...180`c_2@
;
\draw[cyan] (a)
\eee0,1.2`aa`S@
;
\draw[dashed,cyan] (a)--(aa)
\xxx`\tcyan\alpha@
;
\draw[dashed,red] (o)
\zzz1,1`r...swap`\tred\beta...[inner sep=-1.5pt]0`b@
node[circle,draw,solid,minimum size=2.7ex]{}
;
}+\treeo{
\oo`-90`c_1@
\draw (o)
\zzz0,1`a...`\tred\beta\tcyan\alpha\gamma...90`c_2@
;
\draw[dashed,red] (a)
\zzz-1,1`al...`\tred\beta...[inner sep=-1.5pt]180`b@
node[circle,draw,solid,minimum size=2.7ex]{}
;
\draw[cyan] (a)
\eee1,1`ar`S@
;
\draw[dashed,cyan] (a)--(ar)
\xxx swap`\tcyan\alpha@
;
}\\
={}&\sum_{v\in\ver(U)}\sum_{v'\in\ver(U)}T\rjt{\beta}{v'}(S
\rjt{\alpha}{v}U)\\
={}&T\rhd_\beta(S\rhd_\alpha U)-(T\rhd_\beta S)\rhd_{\beta\alpha}U.
\end{align*}
\end{exam}

Let $X$ be a set and let $\Omega$ be a commutative semigroup.
Denote by
\begin{equation*}
\widetilde{\cal{T}}(X,\Omega):=\calt\ot\bfk\Omega.
\mlabel{eq:cal}
\end{equation*}
The following result precises the link between pre-Lie family algebras and ordinary pre-Lie algebras.
\begin{theorem}\mlabel{thm:pre}
$(\bfk\widetilde{\cal{T}}(X,\Omega),\rhd)$ is a pre-Lie algebra, if and only if, $(\bfk\calt, (\rhd_\omega)_{\omega\in\Omega})$ is a pre-Lie family algebra, where
\begin{equation}
(S\ot\alpha)\rhd(T\ot\beta):=(S\rhd_\alpha T)\ot\alpha\beta,\,\text{ for }\, S,T\in \calt\,\text{ and }\, \alpha,\beta\in\Omega.
\mlabel{eq:bullet}
\end{equation}
\end{theorem}
\begin{proof}
For $S,T,U\in\calt$ and $\alpha,\beta,\gamma\in\Omega$, we have
\begin{align*}
&(T\ot\alpha)\rhd\Big((S\ot\beta)\rhd(U\ot\gamma)\Big)
-\Big((T\ot\alpha)\rhd(S\ot\beta)\Big)\rhd(U\ot\gamma)\\
={}&(T\ot\alpha)\rhd\Big((S\rhd_\beta U)\ot\beta\gamma\Big)-\Big((T\rhd_\alpha S)\ot\alpha\beta\Big)\rhd(U\ot\gamma)\quad(\text{by Eq.~(\mref{eq:bullet})})\\
={}&\Big(T\rhd_\alpha(S\rhd_\beta U)\Big)\ot\alpha\beta\gamma-\Big((T\rhd_\alpha S)\rhd_{\alpha\beta}U\Big)\ot\alpha\beta\gamma\quad(\text{by Eq.~(\mref{eq:bullet})})\\
={}&\Big(T\rhd_\alpha(S\rhd_\beta U)-(T\rhd_\alpha S)\rhd_{\alpha\beta}U\Big)\ot\alpha\beta\gamma\\
={}&\Big(S\rhd_\beta(T\rhd_\alpha U)-(S\rhd_\beta T)\rhd_{\beta\alpha}U\Big)\ot\beta\alpha\gamma\quad(\text{by Eq.~(\mref{eq:pre})})\\
={}&(S\ot\beta)\rhd\Big((T\rhd_\alpha U)\ot\alpha\gamma\Big)-\Big((S\rhd_\beta T)\ot\beta\alpha\Big)\rhd(U\ot\gamma)\\
={}&(S\ot\beta)\rhd\Big((T\ot\alpha)\rhd(U\ot\gamma)\Big)
-\Big((S\ot\beta)\rhd(T\ot\alpha)\Big)\rhd(U\ot\gamma),
\end{align*}
as required. Conversely, we obtain
\begin{align*}
&\Big(T\rhd_\alpha(S\rhd_\beta U)-(T\rhd_\alpha S)\rhd_{\alpha\beta}U\Big)\ot\alpha\beta\gamma\\
={}&(T\ot\alpha)\rhd\Big((S\ot\beta) \rhd(U\ot\gamma)\Big)-\Big((T\ot\alpha)\rhd(S\ot\beta)\Big)\rhd(U\ot\gamma)\\
={}&(S\ot\beta)\rhd\Big((T\ot\alpha)\rhd(U\ot\gamma)\Big)
-\Big((S\ot\beta)\rhd(T\ot\alpha)\Big)\rhd(U\ot\gamma)\\
={}&(S\ot\beta)\rhd\Big((T\rhd_\alpha U)\ot\alpha\gamma\Big)-\Big((S\rhd_\beta T)\ot\beta\alpha\Big)\rhd(U\ot\gamma)\\
={}&\Big(S\rhd_\beta(T\rhd_\alpha U)\Big)\ot\beta\alpha\gamma
-\Big((S\rhd_\beta T)\rhd_{\beta\alpha}U\Big)\ot\beta\alpha\gamma\\
={}&\Big(S\rhd_\beta(T\rhd_\alpha U)-(S\rhd_\beta T)\rhd_{\beta\alpha}U\Big)\ot\beta\alpha\gamma.
\end{align*}
Since $\Omega$ is a commutative semigroup, we have
\[T\rhd_\alpha(S\rhd_\beta U)-(T\rhd_\alpha S)\rhd_{\alpha\beta}U=S\rhd_\beta(T\rhd_\alpha U)-(S\rhd_\beta T)\rhd_{\beta\alpha}U. \qedhere \]
\end{proof}
\subsection{Free pre-Lie family algebras}
In this subsection, we show the freeness of the pre-Lie family algebras defined in Paragraph \ref{sub:dd} above.
Let $\Omega$ be a commutative semigroup. Let $(L,(\rhd_\omega)_{\omega\in\Omega})$ be any pre-Lie family algebra. For $x_1,x_2,y\in L$ and $\alpha,\beta\in\Omega,$ define
\begin{equation}
(x_1x_2)\rhd_{\alpha,\beta}y:=x_1\rhd_\alpha(x_2\rhd_\beta y)-(x_1\rhd_\alpha x_2)\rhd_{\alpha\beta}y.
\mlabel{eq:decom1}
\end{equation}
This is symmetric in $(x_1,\alpha)$ and $(x_2,\beta)$ by definition of a pre-Lie family algebra.

Now we define $(x_1\cdots x_n)\dd\omega_1,\ldots,\omega_n;y$, where $x_1\cdots x_n$ is a monomial of $T(L)$ and $y\in L.$
\begin{defn}
For $x_1\cdots x_n, y\in L$ and $\omega_1,\ldots,\omega_n\in\Omega$, we define recursively multilinear maps
\begin{align*}
\dd\omega_1,\ldots,\omega_n;:L^{\ot n}\ot L& \ra L\\
x_1\cdots x_n\ot y& \mapsto (x_1\cdots x_n)\dd\omega_1,\ldots,\omega_n;y
\end{align*}
in the following way: $1\dd\omega; y:=y$ and we define
\begin{align}
&(x_1\cdots x_n)\rhd_{\omega_1,\ldots,\omega_n} y\nonumber\\
:={}&x_1\rhd_{\omega_1}\Big((x_2\cdots x_n)\rhd_{\omega_2,\ldots,\omega_n} y\Big)-\sum^n_{j=2}\Big(x_2\cdots x_{j-1}(x_1\rhd_{\omega_1}x_j)x_{j+1}\cdots x_n\Big)\dd\omega_2,\ldots,\omega_{j-1},\omega_1\omega_j,\omega_{j+1},\ldots,\omega_n; y.
\mlabel{eq:sym}
\end{align}
\end{defn}

\begin{prop}
Let $\Omega$ be a commutative semigroup. Let $(L,(\rhd_\omega)_{\omega\in\Omega})$ be any pre-Lie family algebra. The expression $(x_1\cdots x_n)\rhd_{\omega_1,\ldots,\omega_n} y$ defined in Eq.~(\mref{eq:sym}) is symmetric in $(x_1,\omega_1),\ldots, (x_n,\omega_n).$ Therefore, it defines a map:
\begin{align*}
\dd;:S^n(L\otimes \bfk\Omega)\ot L& \ra L\\
(x_1\otimes\omega_1\cdots x_n\otimes\omega_n)\ot y& \mapsto (x_1\cdots x_n)\dd\omega_1,\ldots,\omega_n;y.
\end{align*}
\end{prop}

\begin{proof}
The invariance by permutation of the variables $(x_1,\omega_1),\ldots,(x_n,\omega_n)$ is obtained by induction on $n\geq 2$. For the initial step $n=2$, it is exactly the pre-Lie family identity~(\mref{eq:decom1}). Denote by $\hat{x_j}$ and $\hat{\omega_j}$ the elements, which don't appear in the terms. For the induction on $n\geq 3$, we have
\begin{align*}
&(x_1\cdots x_n)\dd\omega_1,\ldots,\omega_n;y\\
={}& x_1\dd\omega_1;\Big((x_2\cdots x_n)\dd\omega_2,\ldots,\omega_n; y\Big)-\sum^n_{j=3}\Big(x_2\cdots x_{j-1}(x_1\dd\omega_1; x_j)x_{j+1}\cdots x_n\Big)\dd\omega_2,\ldots,\omega_{j-1},\omega_1\omega_j,w_{j+1},\ldots,\omega_n; y\\
& -\Big((x_1\dd\omega_1;x_2)x_3\cdots x_n\Big)\dd\omega_1\omega_2,\omega_3,\ldots,\omega_n; y\\
={}& x_1\dd\omega_1;\biggl(x_2\dd\omega_2;\Big((x_3\cdots x_n)\dd\omega_3,\ldots,\omega_n; y\Big)\\
&-\sum^n_{j=3}\Big(x_3\cdots x_{j-1}(x_2\dd\omega_2; x_j)x_{j+1}\cdots x_n\Big)\dd\omega_3,\ldots,\omega_{j-1},\omega_2\omega_j,\ldots,\omega_n; y\biggl)\\
&\hspace{2cm}\text{ (expand the first term by equation Eq.~(\mref{eq:sym}))}\\
&-\sum^n_{j=3}\Big(x_2\cdots x_{j-1}(x_1\dd\omega_1; x_j)x_{j+1}\cdots x_n\Big)
\dd\omega_2,\ldots,\omega_{j-1},\omega_1\omega_j,\ldots,\omega_n;y\\
&-\Big((x_1\dd\omega_1;x_2)x_3\cdots x_n\Big)\dd\omega_1\omega_2,\omega_3,\ldots,\omega_n;y\\
={}& x_1\dd\omega_1;\Big(x_2\dd\omega_2;\Big((x_3\cdots x_n)\dd\omega_3,\ldots,\omega_n;y\Big)\Big)\\
&-\sum^n_{j=3}x_1\dd\omega_1;\biggl(\Big(x_3\cdots x_{j-1}(x_2\dd\omega_2; x_j)x_{j+1}\cdots x_n\Big)\dd\omega_3,\ldots,\omega_{j-1},\omega_2\omega_j,\ldots,\omega_n;y\biggl)\\
&-\sum^n_{j=3}\Big(x_2\cdots x_{j-1}(x_1\dd\omega_1;x_j)x_{j+1}\cdots x_n\Big)\dd\omega_2,\ldots,\omega_{j-1},\omega_1\omega_j,\ldots,\omega_n;y\\
& -\Big((x_1\dd\omega_1;x_2)x_3\cdots x_n\Big)\dd\omega_1\omega_2,\omega_3,\ldots,\omega_n;y\\
={}& x_1\dd\omega_1;\biggl(x_2\dd\omega_2;\Big((x_3\cdots x_n)\dd\omega_3,\ldots,\omega_n;y\Big)\biggl)\\
&-\sum^n_{j=3}x_1\dd\omega_1;\biggl(\Big((x_2\dd\omega_2;x_j)x_3\cdots \hat{x_j}\cdots x_n\Big)\dd\omega_2\omega_j,\omega_3,\ldots,\hat{\omega_j},\ldots,\omega_n;y\biggl)\\
&-\sum^n_{j=3}\Big((x_1\dd\omega_1;x_j)x_2\cdots \hat{x_j}\cdots x_n\Big)\dd\omega_1\omega_j,\omega_2,\ldots,\hat{\omega_j},\ldots,\omega_n;y\\
&-\Big((x_1\dd\omega_1;x_2)x_3\cdots x_n\Big)\dd\omega_1\omega_2,\omega_3,\ldots,\omega_n;y\quad\text{(by the induction hypothesis)}\\
={}& x_1\dd\omega_1;\biggl(x_2\dd\omega_2;\Big((x_3\cdots x_n)\dd\omega_3,\ldots,\omega_n;y\Big)\biggl)\\
&-\sum^n_{j=3}\Big(x_1(x_2\dd\omega_2;x_j)x_3\cdots\hat{x_j}\cdots x_n\Big)\dd\omega_1,\omega_2\omega_j,\omega_3,\ldots,\hat{\omega_j},\ldots,\omega_n;y\\
&-\sum^n_{\underset{j\neq k}{j,k=3}}\Big((x_2\dd\omega_2;x_j)(x_1\dd\omega_1;x_k)x_3\cdots\hat{x_j}\cdots \hat{x_k}\cdots x_n\Big)\dd\omega_2\omega_j,\omega_1\omega_k,\omega_3,\ldots,\hat{\omega_j},
\cdots,\hat{\omega_k},\ldots,\omega_n;y\\
&-\sum^n_{j=3}\biggl(x_3\cdots x_{j-1}\Big(x_1\dd\omega_1;(x_2\dd\omega_2;x_j)\Big)\cdots x_n\biggl)\dd\omega_3,\ldots,\omega_{j-1},\omega_1\omega_2\omega_j,
\cdots,\hat{\omega_j},\ldots,\omega_n; y\\
&-\sum^n_{j=3}\Big(x_2(x_1\dd\omega_1;x_j)x_3\cdots \hat{x_j}\cdots x_n\Big)\dd\omega_2,\omega_1\omega_j,\omega_3,\ldots,\hat{\omega_j},\ldots,\omega_n;y\\
&-\Big((x_1\dd\omega_1;x_2)x_3\cdots x_n\Big)\dd\omega_1\omega_2,\omega_3,\ldots,\omega_n;y\quad\text{(expand the second term of the last equation)}
\end{align*}
In the above equality, the sum of the second term and the fifth term is obviously symmetric in $(x_1,\omega_1)$ and $(x_2,\omega_2)$. The third term is also obviously symmetric in $(x_1,\omega_1)$ and $(x_2,\omega_2)$ by itself. We are left to see the remaining three terms, that is, the first term, the fourth term and the sixth term. By the pre-Lie family relation, we have
\begin{align*}
& \text{ $1$th term + $4$th term + $6$th term }\\
={}& (x_1x_2)\dd\omega_1,\omega_2;\Big((x_3\cdots x_n)\dd\omega_3,\ldots,\omega_n;y\Big)
+(x_1\dd\omega_1;x_2)\dd\omega_1\omega_2;\Big((x_3\cdots x_n)\dd\omega_3,\ldots,\omega_n;y\Big)\\
&\hspace{3cm}\text{(expand the first term of last equation)}\\
&-\sum^n_{j=3}\biggl(x_3\cdots x_{j-1}\Big(x_1\dd\omega_1;(x_2\dd\omega_2;x_j)\Big)\cdots x_n\biggl)
\dd\omega_3,\ldots,\omega_{j-1},\omega_1\omega_2\omega_j,
\cdots,\hat{\omega_j},\ldots,\omega_n;y\\
&-(x_1\dd\omega_1;x_2)\dd\omega_1\omega_2;\Big((x_3\cdots x_n)\dd\omega_3,\ldots,\omega_n;y\Big)\\
&+\sum^n_{j=3}\biggl(x_3\cdots x_{j-1}\Big((x_1\dd\omega_1; x_2)\dd\omega_1\omega_2;x_j\Big)\cdots x_n\biggl)
\dd\omega_3,\ldots,\omega_{j-1},
\omega_1\omega_2\omega_j,\ldots,\hat{\omega_j},\ldots,\omega_n;y\\
&\hspace{3cm}\text{ (expand the sixth term by Eq.~(\mref{eq:sym}))}\\
={}& (x_1x_2)\dd\omega_1,\omega_2;\Big((x_3\cdots x_n)\dd\omega_3,\ldots,\omega_n;y\Big)\\
&-\sum^n_{j=3}\biggl(x_3\cdots x_{j-1}\Big((x_1x_2)\dd\omega_1,\omega_2;x_j\Big)\cdots x_n\biggl)\dd\omega_3,\ldots,
\omega_{j-1},\omega_1\omega_2\omega_j,\ldots,\hat{\omega_j},\ldots,\omega_n;y.\\
&\hspace{3cm}\text{ (by the pre-Lie family relation)}
\end{align*}
Hence the sum is obviously symmetric in $(x_1,\omega_1)$ and $(x_2,\omega_2)$. By the induction hypothesis, $(x_1\cdots x_n)\dd\omega_1,\ldots,\omega_n;y$ is symmetric in the $n-1$ variables $(x_2,\omega_2),\ldots,(x_n,\omega_n)$. So we get the announced invariance. This completes the proof.
\end{proof}
\ignore{
\begin{defn}
Let $\Omega$ be a commutative semigroup and let $L$ be a pre-Lie family algebra. For $x,y_1,\ldots, y_n\in L$ and $\omega\in\Omega$, define maps $\rhd_\omega:L\ot S^n(L)\ra S^n(L)$ in the following way
\begin{equation}
x\rhd_\omega 1:=0\,\text{ and }\, x\rhd_\omega (y_1\cdots y_n):=\sum_{1\leq i\leq n}y_1\cdots(x\rhd_\omega y_i)\cdots y_n.
\mlabel{eq:mult}
\end{equation}
\end{defn}
}
\begin{defn}\mlabel{defn:bplus}
Let $X$ be a set and let $\Omega$ be a commutative semigroup. Let $x\in X, T_1\cdots T_n\in\calt,$ $\omega_1,\ldots,\omega_n\in\Omega.$ Denote by
$$
B^{+}_{x,\omega_1,\ldots,\omega_n}(T_1\cdots T_n)=\treeo{
\oo`-90`x@
\draw (o)
\eee-1,1`l`T_1@
(o)--(l)\xxx`\omega_1@
;
\draw (o)
\eee1,1`r`T_n@
(o)--(r)\xxx swap`\omega_n@
;
\node at (0,1) {$\dots$};
}$$
the $\Omega$-typed $X$-decorated tree obtained by grafting $T_1\cdots T_n$ on a common root decorated by $x$, the edge between this root and the root of $T_i$ being of type $\omega_i$ for any $i$. This defines maps
$$B^{+}_{x,\omega_1,\ldots,\omega_n}:T^n\big((\calt\bigr)\ra \calt.$$
\end{defn}

\begin{lemma}\mlabel{lem:bplus}
Let $X$ be a set and let $\Omega$ be a commutative semigroup.
For any $x\in X$ and $\omega_1, \cdots,\omega_n\in\Omega$, we have 
$$B^{+}_{x,\omega_1,\ldots,\omega_n}(T_1\cdots T_n)=(T_1\cdots T_n)
 \rhd_{\omega_1,\ldots,\omega_n}\bullet_x.$$
\end{lemma}

\begin{proof}
Let $F=T_1\cdots T_n.$ We proceed by induction on $n\geq 1$. If $n=1$, then $F=T_1$ and $T_1 \rhd_\omega \bullet_x=\treeo{
\oo`-90`x@
\draw (o)
\eee0,1`l`T_1@
(o)--(l)\xxx swap`\omega@
;
}=B^+_{x,\omega}(T_1)$. Let us assume that the result holds for $n-1$, with $n\geq 2.$ We can write $F=T_1F'$ with length $l(F')=n-1$. Then:
\begin{align*}
F \rhd _{\omega_1,\ldots\omega_n} \bullet_x&=(T_1F')\dd\omega_1,\ldots,\omega_n; \bullet_x\\
&=T_1 \rhd_{\omega_1} (F' \rhd_{\omega_2,\ldots,\omega_n}  \bullet_x)-\sum_{j=2}^n \Bigl((T_2\cdots T_{j-1}(T_1 \rhd_{\omega_1} T_j)\cdots T_n\Bigr) \rhd_{\omega_2,\ldots,\omega_1\omega_j,\ldots,\omega_n}  \bullet_x\quad\text{(by Eq.~(\mref{eq:sym}))}\\
&=T_1 \rhd _{\omega_1} B^+_{x,\omega_1,\ldots,\omega_n}(F')-\sum_{j=2}^n B^+_{x,\omega_2,\ldots,\omega_1\omega_j,\ldots,\omega_n}\Bigl((T_2\cdots T_{j-1}(T_1 \rhd_{\omega_1} T_j)\cdots T_n\Bigr)\\
&\hskip 90mm (\text{by the induction hypothesis})\\
&=T_1 \rhd _{\omega_1} B^+_{x,\omega_1,\ldots,\omega_n}(T_2\cdots T_n)-\sum_{j=2}^n B^+_{x,\omega_2,\ldots,\omega_1\omega_j,\ldots,\omega_n}\Bigl((T_2\cdots T_{j-1}(T_1 \rhd_{\omega_1} T_j)\cdots T_n\Bigr)\\
\ignore{
&=T_1\rhd_\alpha \treeo{
\oo`-90`x@
\draw (o)
\eee0,1`l`F'@
(o)--(l)\xxx`\beta@
;
}
-\treeo{
\oo`-90`x@
\draw (o)
\eee0,1`l`T_1\rhd_\alpha F'@
(o)--(l)\xxx`\alpha\beta@
;
}\\
&=\treeo{
\oo`-90`x@
\draw (o)
\eee0,1`l`T_1\rhd_\alpha F'@
(o)--(l)\xxx`\alpha\beta@
;
}+\treeo{
\oo`-90`x@
\draw (o)
\eee-1,1`l`T_1@
(o)--(l)\xxx`\alpha@
;
\draw (o)
\eee1,1`r`F'@
(o)--(r)\xxx swap`\beta@
;
}
-\treeo{
\oo`-90`x@
\draw (o)
\eee0,1`l`T_1\rhd_\alpha F'@
(o)--(l)\xxx`\alpha\beta@
;
}\\
}
&=\treeo{
\oo`-90`x@
\draw (o)
\eee-1,1`l`T_1@
(o)--(l)\xxx`\omega_1@
;
\draw (o)
\eee1,1`r`F'@
(o)--(r)\xxx swap`\omega_2,\ldots,\omega_n@
;
}\\
&=B^+_{x,\omega_1,\ldots,\omega_n}(T_1F')=B^+_{x,\omega_1,\ldots,\omega_n}(F).
\end{align*}
So the result holds for all $n\geq 0.$
\end{proof}

Let $j: X\ra \calt, x\mapsto \bullet_x$ be the standard embedding map.
\begin{theorem}
Let $X$ be a set and let $\Omega$ be a commutative semigroup. Then $(\calt,(\rhd_\omega)_{\omega\in\Omega})$, together with the map $j$, is the free pre-Lie family algebra on $X$.
\mlabel{thm:free}
\end{theorem}
\begin{proof}
 Let $(A,(\rhd'_\omega)_{\omega\in\Omega})$ be a pre-Lie family algebra. Choose a set map $\psi:X\ra A$, and use the shorthand notation $a_x$ for $\psi(x)$.\\
\noindent{\bf Existence:} Define a linear map $\phi:\calt\ra A$ as follows. We define $\phi(T)$ by induction on $|T|\geq 1$. For the initial step $|T|=1,$ we have $T=\bullet_x$ and define $$\phi(\bullet_x):=a_x.$$
For the induction step $|T|\geq 2,$ let $T=B^{+}_{x,\omega_1,\ldots,\omega_n}(T_1\cdots T_n)=\treeo{
\oo`-90`x@
\draw (o)
\eee-1,1`l`T_1@
(o)--(l)\xxx`\omega_1@
;
\draw (o)
\eee1,1`r`T_n@
(o)--(r)\xxx swap`\omega_n@
;
\node at (0,1) {$\dots$};
}$, and define
\begin{align}
\phi\Big(B^{+}_{x,\omega_1,\ldots,\omega_n}(T_1\cdots T_n)\Big)&=\phi\Big(\treeo{
\oo`-90`x@
\draw (o)
\eee-1,1`l`T_1@
(o)--(l)\xxx`\omega_1@
;
\draw (o)
\eee1,1`r`T_n@
(o)--(r)\xxx swap`\omega_n@
;
\node at (0,1) {$\dots$};
}\Big)\nonumber\\
&=\phi\Big((T_1\cdots T_n)\rhd_{\omega_1,\ldots,\omega_n}\bullet_x\Big)\nonumber\\
&:=\Big(\phi(T_1)\cdots\phi(T_n)\Big)\rhd'_{\omega_1,\ldots,\omega_n}a_x.
\mlabel{eq:phi}
\end{align}
Let $T,T'\in\calt$ and $\omega\in\Omega$. We are left to prove that
$$\phi(T \rhd _\omega T')=\phi(T) \rhd' _\omega\phi(T')$$
 by induction on  $|T'|\geq 1.$ For the initial step $|T'|=1$, we have $T'= \bullet_x.$ Then $T \rhd _\omega T'=B^+_{x,\omega}(T),$ so we have
$$\phi(T \rhd _\omega T')=\phi(B^+_{x,\omega}(T))=\phi(T)\rhd'_\omega a_x=\phi(T)\rhd'_\omega\phi(T').$$
For the induction step $|T'|\geq 2,$ let
$$T'=B^+_{x,\omega_1,\ldots,\omega_n}(T'_1\cdots T'_n)=\treeo{
\oo`-90`x@
\draw (o)
\eee-1,1`l`T'_1@
(o)--(l)\xxx`\omega_1@
;
\draw (o)
\eee1,1`r`T'_n@
(o)--(r)\xxx swap`\omega_n@
;
\node at (0,1) {$\dots$};
}.$$
Then
\begin{align*}
T\rhd_\omega T'&=T\rhd_\omega B^+_{x,\omega_1,\ldots,\omega_n}(T'_1\cdots T'_n)\\
&=\treeo[y=1.2cm]{
\oo`-90`x@
\draw (o)
\eee-0.8,1`l`T_1'@
(o)--(l)\xxx swap`\omega_1@
;
\draw (o)
\eee0.8,1`r`T_n'@
(o)--(r)\xxx swap`\omega_n@
;
\node at (0,1) {$\dots$};
\draw (o)
\eee-1.8,1`L`T@
;
\draw[dashed]
(o)--(L)\xxx`\omega@
;
}+\sum_{j=1}^n\treeo[y=1.2cm]{
\oo`-90`x@
\draw (o)
\eee-1.8,1`l`T_1'@
(o)--(l)\xxx`\omega_1@
;
\draw (o)
\eee1.8,1`r`T_n'@
(o)--(r)\xxx swap`\omega_n@
;
\draw (o)
\eee0,1`a`T\rhd_\omega T_j'@
;
\draw[dashed]
(o)--(a)\xxx pos=0.8`\omega\omega_j@
node[pos=0.55]{$\,\dots\,\dots$}
;
}\\
&=B^+_{x,\omega,\omega_1,\ldots,\omega_n}(TT'_1\cdots T'_n)
+\sum_{j=1}^n B^+_{x,\omega_1,\ldots,\omega\omega_j,\ldots,\omega_n}(T'_1\cdots
(T\rhd_\omega T'_j)\cdots T'_n).
\end{align*}
So we have
\begin{align*}
&\phi(T\rhd_\omega T')\\
&=\phi\biggl(B^+_{x,\omega,\omega_1,\ldots,\omega_n}(TT'_1\cdots T'_n)
+\sum_{j=1}^n B^+_{x,\omega_1,\ldots,\omega\omega_j,\ldots,\omega_n}(T'_1\cdots
(T\rhd_\omega T'_j)\cdots T'_n)\biggl)\\
&=\Big(\phi(T)\phi(T'_1)\cdots\phi(T'_n)\Big)\dd'\omega,\omega_1,\ldots,\omega_n;a_x
+\sum^n_{j=1}\Big(\phi(T'_1)\cdots \phi(T\rhd_\omega T'_j)\cdots \phi(T'_n)\Big)\dd'\omega_1,\ldots,\omega\omega_j,\ldots,\omega_n;a_x\\
&\hspace{4cm}\text{(by Eq.~(\mref{eq:phi}))}\\
&=\Big(\phi(T)\phi(T'_1)\cdots\phi(T'_n)\Big)\dd'\omega,\omega_1,\ldots,\omega_n;a_x
+\sum^n_{j=1}\bigg(\phi(T'_1)\cdots \Big(\phi(T)\rhd'_\omega \phi(T'_j)\Big)\cdots \phi(T'_n)\bigg)\dd'\omega_1,\ldots,\omega\omega_j,\ldots,\omega_n;a_x\\
&\hspace{4cm}\text{(by the induction hypothesis)}\\
%
%
&=\phi(T)\rhd'_\omega\bigg(\Big(\phi(T'_1)\cdots\phi(T'_n)\Big)
\rhd'_{\omega_1,\ldots,\omega_n}a_x\bigg)\quad\text{(by Eq.~(\mref{eq:sym}))}\\
&=\phi(T)\rhd'_\omega\phi(T').
\end{align*}
This completes the proof.
\end{proof}
\section{The pre-Lie family operad}
\mlabel{sec:oper}
We generalize the description of the pre-Lie operad in terms of labeled rooted trees by Chapoton and Livernet~\cite{CL} to a description of the pre-Lie family operad in terms of typed labeled rooted trees.
\subsection{The operad of pre-Lie family algebras}
We now describe an operad in the linear species framework. We refer to~\cite{BD, LoVa, MSS, MM} for notations and definitions on operads.
In fact, a pre-Lie family algebra is an algebra over a binary quadratic operad, denoted by $\pl$.\\

\noindent From now on, let $\cal{C}$ be a symmetric tensor category~\cite{Mac}.
\begin{defn}
 A species in the category $\cal{C}$ is a contravariant functor $E$ from the category of finite sets $E_{fin}$ with bijections to
$\cal{C}$. Thus, a species $E$ provides an object $E_A$ for any finite set $A$ and an isomorphism $E_\phi:E_B\ra E_A$ for any bijection $\phi:A\ra B.$
\mlabel{defn:spe}
\end{defn}
\begin{defn}
A morphism of species between $F$ and $G$ is a natural transformation $\psi:F\ra G$.
So for any finite sets $A$ and $B$ of the same cardinal and any bijection $\phi$ from $A$ to $B$, the following diagram commutes:
$$
\xymatrix{
  F_B \ar[d]_{\psi_B} \ar[r]^{F_\phi}
                & F_A \ar[d]^{\psi_A} \\
  G_B
                & G_A
                \ar@{<-}[l]^-{G_\phi}            }
$$
\end{defn}

\noindent Recall \cite[Paragraph 3.2.4]{MM} that the monoidal structure $\circ$ on species is given by
\begin{equation}
(\mathcal P\circ \mathcal Q)_A:=\bigoplus_{\pi\hbox{ \tiny partition of }A}\mathcal P_\pi\otimes\bigotimes_{B\in\pi}\mathcal Q_B.
\end{equation}
\begin{defn}
An operad $\cal{P}=(\cal{P}, \gamma,\eta)$ is a species $A\mapsto \cal P_A$ endowed with a monoid structure, i.e. an associative composition map $\gamma:\cal{P}\circ\cal{P}\ra\cal{P}$ and a unit map $\eta:{\bf I}\ra \cal{P}$ where ${\bf I}$ is the species defined by ${\bf I}_A=0$ for $|A|\neq 1$ and ${\bf I}_{\{*\}}=1_{\mathcal C}$.
\end{defn}
Partial compositions give an equivalent and simpler way to define operads, because it reduces the operad multiple substitution to binary operations that frequently are easier to define.

\begin{defn}
Let $\cal P$ be an operad and let $\mu\in\calp(A), \nu\in\calp(B)$ be two operations, where $A$ and $B$ are two finite sets.
Define the partial compositions
\begin{align*}
\circ_a: \calp(A)\ot\calp(B) & \ra\calp(A\sqcup B\backslash\{a\}),\,\text{ for }\, a\in A.\\
\mu\circ_a\nu &:= \gamma\bigl(\mu;(\alpha_x)_{x\in A}\bigr)
\end{align*}
with $\alpha_a=\nu$ and $\alpha_x=\hbox{id}$ for $x\neq a$. There are two different cases for two-stage partial compositions, depending on the relative positions of the two insertions. Associativity of the composition in an operad leads to two associativity axioms for the partial compositions, one for each case (sequential or parallel):
\begin{equation*}
\left\{
\begin{array}{ll}
\text{(I)}\quad (\lambda \circ_a \mu)\circ_b\nu=\lambda \circ_a(\mu\circ_b\nu), & \text{ for } a\in A, b\in B; \\
\text{(II)}\quad (\lambda\circ_a \mu)\circ_{a'}\nu=(\lambda\circ_{a'} \nu)\circ_a\mu, & \text{ for }  a,a'\in A,
\end{array}
\right .
\end{equation*}
for any $\lambda\in\calp(A),\mu\in\calp(B),\nu\in\calp(C).$ Relation (I) is called the {\bf sequential composition} axiom and relation (II) is called the {\bf parallel composition} axiom. The unit element $\id\in\calp(1)$ satisfies
$$\text{(III)}\quad \id\circ_{\{*\}}\mu=\mu\hbox{ and }\mu\circ_b\id=\mu$$
for any $b\in B$. Conversely, given a family of partial compositions verifying the three axioms above, we can recover the global composition by choosing any enumeration $b_1,\ldots,b_n$ of the finite set $B$, and by setting
$$\gamma(\mu;\alpha_1,\ldots,\alpha_n):=\Bigl(\cdots\bigl((\mu\circ_{b_1}\alpha_1)\circ_{b_2}\alpha_2\bigr)\cdots\Bigr)\circ_{b_n}\alpha_n.$$
This does not depend on the choice of the enumeration by virtue of the iterated Axiom (II).
\end{defn}

 Denote by $|A|$ the cardinality of the finite set $A$. Let $\ff$ be the free operad generated by the regular representation of $S_2$ on the species $E$ defined as follows:
\begin{equation*}
E_A=
\left\{
\begin{array}{ll}
0, & \text{ if }|A|\neq 2;\\
{\mathbf k}\Omega\ot\bfk S_2, & \text{ if }  |A|=2.
\end{array}
\right .
\end{equation*}

We choose a basis $(\mu_\omega)_{\omega\in\Omega}$ of ${\mathbf k}\Omega$. The space $(\ff)_A$ is the linear span of planar binary trees where each internal vertex $v$ is decorated by an element of $E_{\hbox{\tiny In}(v)}$ and leaves are labelled by the elements of the set $A$. Here $\hbox{In}(v)$ stands for the set of incoming edges of vertex $v$.
A basis of $\ff(A)$, as a vector space, is given by $|A|-2$ formal partial compositions of the binary elements $\rhd_\omega$ and $\tau\rhd_\omega$.
Let $R$ be the $S_3$-submodule of $\ff(A)$ generated by the relations
\begin{equation}\label{relations-operad}
r=\rhd_\alpha\circ_b\rhd_\beta-\rhd_{\alpha\beta}\circ_a\rhd_\alpha
-\tau_{12}(\rhd_\beta\circ_b\rhd_\alpha-\rhd_{\beta\alpha}\circ_a\rhd_\beta),
\end{equation}
Choose an enumeration $\{a,b\}$ of the set $A$, and an enumeration $\{1,2\}$ of a second disjoint copy of the set $A$. Note that $\rhd_\alpha\circ_b\rhd_\beta$ and $\rhd_\beta\circ_b\rhd_\alpha$ belong to $\pl_{\{a,1,2\}}$ whereas $\rhd_{\alpha\beta}\circ_a\rhd_\beta$ and  $\rhd_{\beta\alpha}\circ_a\rhd_\alpha$ belong to $\pl_{\{1,2,b\}}$. We identify both three-element sets by means of the bijection
\begin{equation*}
\begin{pmatrix}
a & 1 & 2 \\
1 & 2 & b
\end{pmatrix}
\end{equation*}
in order to make Equation \eqref{relations-operad} consistent. Then $\pl=\ff/(R),$ where $(R)$ denotes the operadic ideal of $\ff$ generated by $R$.
\subsection{The operad of typed labeled rooted trees}
Let $\tm(A)$ be the set of typed labeled rooted trees, whose leaves are labelled by the elements of set $A$ and the edges are decorated by elements of $\Omega.$ We denote by $\tr(A)$ the free \bfk-vector space generated by the $\Omega$-typed $A$-labeled rooted trees $\tm(A)$. We can endow $\tr$ with a linear operad structure as follows. Recall that $\hbox{In}(v)$ stands for the set of incoming edges at the vertex $v$ of $T$.

 \begin{defn}
 Let $A, B$ be two sets. We define the composition of $T$ and $U$ along the vertex $v$ of $T$ by $\circ_v:\tr(A)\ot\tr(B)\ra\tr(A\sqcup B\backslash\{v\}),$ for $T\in\tm(A)$ and $U\in\tm(B)$
\begin{equation}
T\circ_v U:=\sum_{f:\hbox{\tiny In}(v)\ra\ver{(U)}}T\circ^f_v U,
\mlabel{eq:comp}
\end{equation}
where $T\circ^f_v U$ is the typed labeled rooted tree of $\tm(A\sqcup B\backslash\{v\})$ obtained by
\begin{itemize}
\item replacing the vertex $v$ of $T$ by the tree $U$,
\item connecting each edge $a$ in $\hbox{In}(v)$ at the vertex $f(a)$ of $U$,
\item multiplying by $\omega_a$ the type of any edge below the vertex $f(a)$, where $\omega_a$ is the type of the edge $a$.
\end{itemize}
\mlabel{defn:comp}
\end{defn}

\noindent For better understanding, we give an example.
\begin{exam}
Let $\Omega$ be a commutative semigroup.
Let
$$
T=\treeo{
\oo`-90`x_7@
\draw (o)
\zzz-1,1`l...`\alpha_3...[inner sep=0pt]-135`\tred{v}@
\zzz-0.5,1`ll...`\alpha_1...90`x_1@
;
\draw (l)
\zzz0.5,1`lr...swap`\alpha_2...90`x_2@
;
\draw (o)
\zzz1,1`r...swap`\alpha_4...[inner sep=0pt]-45`x_6@
\zzz0.5,1`rr...swap`\alpha_6...90`x_4@
;
\draw (r)
\zzz-0.5,1`rl...`\alpha_5...90`x_3@
;
},\,\text{ and }\,
U=\treeo{
\oo`-90`y_7@
\draw (o)
\zzz0,1`oo...`\beta_7...[inner sep=0pt]-45`y_3@
\zzz-1,1`l...`\beta_2...180`y_2@
\zzz0.5,1`lr...`\beta_1...90`y_1@
;
\draw (oo)
\zzz1,1`r...swap`\beta_3...[inner sep=0pt]-45`{v'}@
\zzz0.8,1`rr...swap`\beta_6...90`y_6@
;
\draw (r)
\zzz-0.8,1`rl...`\beta_4...90`y_4@
;
\draw (r)
\zzz0,1`ra...swap,pos=0.75`\beta_5...90`y_5@
;
}.$$

We are ready to compute $T\circ_v U$, first we replace the vertex $v$ of $T$ with the tree $U$, then we graft the edges which contain vertices $x_1$ and $x_2$ of $T$ on the vertices of $U$, there are many cases. Here we just write one case, we graft the two edges of $T$ we mentioned above on the left vertex $y_1$ and right vertex $y_6$ of $U$ respectively, that is, $f(\treeoo{
\draw (o)
\zzz0,0.6`l...swap`\alpha_1...90`x_1@
;})=y_1$ and $f(\treeoo{
\draw (o)
\zzz0,0.6`l...swap`\alpha_2...90`x_2@
;})=y_6$. Denote by $T\circ_{v,lr}U$ this particular component of the composition. Then we have
$$
T\circ_{v,lr} U=
\treeo{
\oo`-90`x_7@
\draw (o)
\zzz-1.5,1`l,red...`\alpha_3...180`\tred{y_7}@
;
\draw (o)
\zzz1.5,1`r...swap`\alpha_4...[inner sep=0pt]-45`x_6@
\zzz0.5,1`rr...swap`\alpha_6...90`x_4@
;
\draw (r)
\zzz-0.5,1`rl...`\alpha_5...90`x_3@
;
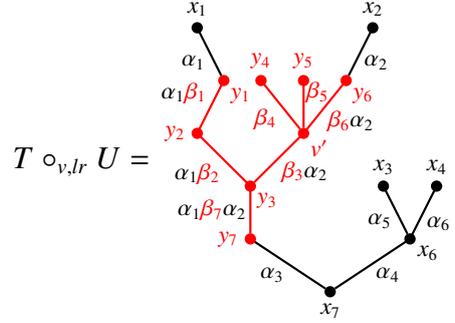
\begin{scope}[red]
\draw (l)
\zzz0,1`uo...`\tblk{\alpha_1}\beta_7\tblk{\alpha_2}...[inner sep=0pt]-45`y_3@
\zzz-1,1`ul...`\tblk{\alpha_1}\beta_2...180`y_2@
\zzz0.5,1`ulr...`\tblk{\alpha_1}\beta_1...[inner sep=0pt]-45`y_1@
;
\draw (uo)
\zzz1,1`ur...swap`\beta_3\tblk{\alpha_2}...[inner sep=0pt]-45`{v'}@
\zzz0.8,1`urr...swap`\beta_6\tblk{\alpha_2}...[inner sep=0pt]-45`y_6@
;
\draw (ur)
\zzz-0.8,1`url...`\beta_4...90`y_4@
;
\draw (ur)
\zzz0,1`ura...swap,pos=0.75`\beta_5...90`y_5@
;
\end{scope}
\draw (ulr)
\zzz-0.5,1`ll...`\alpha_1...90`x_1@
;
\draw (urr)
\zzz0.5,1`lr...swap`\alpha_2...90`x_2@
;
}
.$$
\mlabel{exam:ex}
\end{exam}

\begin{prop}
The species $\tr$ together with the partial compositions $\circ_v$ defined by Eq.~(\mref{eq:comp}) is an operad.
\mlabel{prop:operad}
\end{prop}

\begin{proof}
We adapt the proof from \cite[Theorem 10]{S14}, where the associativity is proved in the slightly different context of a "current-preserving" version of the pre-Lie operad. Let $T\in\tm(A), U\in\tm(B)$ and $W\in\tm(C)$ where $A,B,C$ are three finite sets. First, we prove sequential associativity:
let $v\in T$ and $v'\in U.$ Then we have
 \begin{align*}
(T\circ_v U)\circ_{v'}W&=\sum_{f:\hbox{\tiny In}(v)\ra \ver{(U)}}
(T\circ^f_v U)\circ_{v'}W\\
&=\sum_{f:\hbox{\tiny In}(v)\ra \ver{(U)}}\sum_{g:\widetilde{\hbox{\tiny In}}(v')\ra\ver{(W)}}(T\circ^f_v U)\circ^g_{v'}W,
\end{align*}
where $\widetilde{\hbox{In}}(v')$ stands for the set of incoming edges of $v'$ inside the tree $T\circ_v^fU$. Similarly we have
 \begin{align*}
 T\circ_v(U\circ_{v'}W)&=\sum_{\tilde{g}:\hbox{\tiny In}(v')\ra \ver{(W)}}T\circ_v(U\circ^{\tilde{g}}_{v'}W)\\
 &=\sum_{\tilde{f}:\hbox{\tiny In}(v)\ra \ver{(U\circ^{\tilde{g}}_{v'}W)}}\sum_{\tilde{g}:\hbox{\tiny In}(v')\ra \ver{(W)}}T\circ^{\tilde{f}}_v(U\circ^{\tilde{g}}_{v'}W).
 \end{align*}
 In order to show $(T\circ_v U)\circ_{v'}W=T\circ_v(U\circ_{v'}W),$ we have to prove that there exists a natural bijection $(f,g)\mapsto (\tilde{f},\tilde{g})$ such that
 \begin{equation}
 (T\circ^f_v U)\circ^g_{v'}W=T\circ^{\tilde{f}}_v(U\circ^{\tilde{g}}_{v'}W).
 \mlabel{eq:bi}
 \end{equation}
 Let $f:\hbox{In}(v)\ra\ver{(U)}$ and $g:\widetilde{\hbox{In}}(v')\ra\ver{(W)}$ be two maps. We look for $\tilde{g}:\hbox{ In}(v')\ra\ver{(W)}$ and $\tilde{f}:\hbox{ In}(v)\ra\ver{(U\circ^{\tilde{g}}_{v'}W)}=\ver{(U)}\sqcup \ver{(W)}\backslash\{v'\}$
such that $(T\circ^f_v U)\circ^g_{v'}W=T\circ^{\tilde{f}}_v(U\circ^{\tilde{g}}_{v'}W)$ holds.

Let $a$ be an edge of $U$ arriving at $v'$, thus $a$ is an edge of $T\circ^f_v U$ arriving at $v'$. We get $\tilde{g}(a)=g(a)$, hence $\tilde g$ is the restriction of $g$ to $\hbox{In}(v')$. Similarly we define $\tilde{f}$ in a unique way:
\begin{align*}
\tilde{f}:\hbox{In}(v)&\ra\ver(U\circ^{\tilde{g}}_{v'}W)=\ver{(U)}\sqcup \ver{(W)}\backslash\{v'\}\\
a&\mapsto \tilde{f}(a)=
\left\{
\begin{array}{ll}
f(a), & \text{ if } f(a)\neq v';\\
g(a), & \text{ if } f(a)=v'.
\end{array}
\right .
\end{align*}
Conversely, we assume that we have the pair $(\tilde{f},\tilde{g})$ and look for the pair $(f,g)$ such that Eq.~(\mref{eq:bi}) holds. We have $\tilde{f}:\hbox{In}(v)\ra\ver{(U)}\sqcup \ver{(W)}\backslash\{v'\}$ and $\tilde{g}:\hbox{In}(v')\ra\ver{(W)}.$ We can define
\begin{align*}
f:\hbox{In}(v)&\ra\ver{(U)}\\
a&\mapsto f(a)=\left\{
\begin{array}{ll}
\tilde{f}(a), & \text{ if } \tilde{f}(a)\notin\ver{(W)};\\
v', & \text{ if } \tilde{f}(a)\in \ver{(W)}.
\end{array}
\right .
\end{align*}
and
\begin{align*}
g:\widetilde{\hbox{In}}(v')&\ra\ver{(W)}\\
a&\mapsto g(a)=\left\{
\begin{array}{ll}
\tilde{g}(a), & \text{ if } a \text{ is an edge of }U;\\
\tilde{f}(a), & \text{ if } a \text{ is an edge of }T.
\end{array}
\right .
\end{align*}
The two subjacent trees of $(T\circ^f_v U)\circ^g_{v'}W$ and $T\circ^{\tilde{f}}_v(U\circ^{\tilde{g}}_{v'}W)$ are the same. In both cases an edge in $U$ has its type multiplied by the types of the edges of $T$ arriving above it along the plugging map $f$. For the left-hand side, an edge in $W$ has its type multiplied by the types of the edges of $T\circ^f_v U$ arriving above it along the plugging map $g$. For the right-hand side, an edge in $W$ has its type multiplied by the types of the edges of $U$ arriving above it along the plugging map $\tilde g$, and also mutiplied by  the types of the edges of $T$ arriving above it along the plugging map $\tilde f$. This amounts to the same, due to commutativity of the semigroup $\Omega$, which proves \eqref{eq:bi}.\\

\noindent Second, we prove the parallel  associativity: let $v$ and $v'$ be two disjoint vertices of $T$, then
\begin{align*}
(T\circ_vU)\circ_{v'}W&=\sum_{f:\hbox{\tiny In}(v)\ra\ver{(U)}}(T\circ^f_v U)\circ_{v'}W\\
&=\sum_{f:\hbox{\tiny In}(v)\ra\ver{(U)}}\sum_{g:\hbox{\tiny In}(v')\ra\ver{(W)}}
(T\circ^f_v U)\circ^g_{v'}W.
\end{align*}
Similarly, we have
\begin{align*}
(T\circ_{v'}W)\circ_v U&=\sum_{g:\hbox{\tiny In}(v')\ra\ver{(W)}}(T\circ^{g}_{v'}W)\circ_v U\\
&=\sum_{g:\hbox{\tiny In}(v')\ra\ver{(W)}}\sum_{f:\hbox{\tiny In}(v)\ra\ver{(U)}}
(T\circ^{g}_{v'}W)
\circ^{f}_v U.
\end{align*}
The equality $(T\circ_v U)\circ_{v'}W=(T\circ_{v'} W)\circ_{v} U$ comes from the fact that both sides have the same subjacent tree, which are identically typed by virtue of the commutativity of $\Omega$.
\end{proof}

\noindent For a more intuitive understanding of Proposition~\mref{prop:operad}, we give the following example.
\begin{exam}
Let $\Omega$ be a commutative semigroup. Let
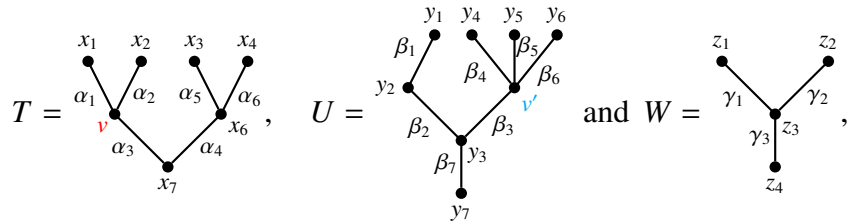
\[
T=\treeo{\oo`-90`x_7@
\draw (o)
\zzz-1,1`l...`\alpha_3...[inner sep=0pt]-135`\tred{v}@
\zzz-0.5,1`ll...`\alpha_1...90`x_1@
;
\draw (l)
\zzz0.5,1`lr...swap`\alpha_2...90`x_2@
;
\draw (o)
\zzz1,1`r...swap`\alpha_4...[inner sep=0pt]-45`x_6@
\zzz0.5,1`rr...swap`\alpha_6...90`x_4@
;
\draw (r)
\zzz-0.5,1`rl...`\alpha_5...90`x_3@
;
},\quad
U=\treeo{\oo`-90`y_7@
\draw (o)
\zzz0,1`oo...`\beta_7...[inner sep=0pt]-45`y_3@
\zzz-1,1`l...`\beta_2...180`y_2@
\zzz0.5,1`lr...`\beta_1...90`y_1@
;
\draw (oo)
\zzz1,1`r...swap`\beta_3...[inner sep=0pt]-45`\tcyan{v'}@
\zzz0.8,1`rr...swap`\beta_6...90`y_6@
;
\draw (r)
\zzz-0.8,1`rl...`\beta_4...90`y_4@
;
\draw (r)
\zzz0,1`ra...swap,pos=0.75`\beta_5...90`y_5@
;
}\,\text{ and }\,
W=\treeo{\oo`-90`z_4@
\draw (o)
\zzz0,1`oo...`\gamma_3...[inner sep=0pt]-45`z_3@
\zzz-1,1`l...`\gamma_1...90`z_1@
;
\draw (oo)
\zzz1,1`r...swap`\gamma_2...90`z_2@
;
},
\]
Here we just illustrate the sequential associativity: $(T\circ_v U)\circ_{v'}W=T\circ_v(U\circ_{v'}W)$ as the parallel associativity is simpler to understand.
But there are many cases when we compute the composition, here we just illustrate a particular case. Let $f:\hbox{In}(v)\ra\ver{(U)}$ and $g:\widetilde{\hbox{In}}(v')\ra\ver{(W)}$. Choose
$$f(\treeoo{
\draw (o)
\zzz0,0.6`l...swap`\alpha_1...90`x_1@
;})=y_1,f(\treeoo{
\draw (o)
\zzz0,0.6`l...swap`\alpha_2...90`x_2@
;})=y_6\,\text{ and }\,g(\treeoo{
\draw (o)
\zzz0,0.6`l...swap`\beta_4...90`y_4@
;})=z_1, g(\treeoo{
\draw (o)
\zzz0,0.6`l...swap`\beta_5...90`y_5@
;})=z_3, g(\treeoo{
\draw (o)
\zzz0,0.6`l...swap`\beta_6...90`y_6@
;})=z_2.$$
We divide into two steps of the left hand side: the first step we replace the vertex $v$ of $T$ with $U$, we denote by $T\circ^f_{v,1}U$ the particular component of the composition, this step we can refer to Example~\mref{exam:ex}. So we have
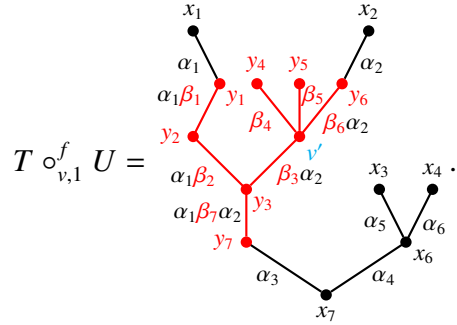
\[
T\circ^f_{v,1}U=
\treeo{\oo`-90`x_7@
\draw (o)
\zzz-1.5,1`l,red...`\alpha_3...180`\tred{y_7}@
;
\draw (o)
\zzz1.5,1`r...swap`\alpha_4...[inner sep=0pt]-45`x_6@
\zzz0.5,1`rr...swap`\alpha_6...90`x_4@
;
\draw (r)
\zzz-0.5,1`rl...`\alpha_5...90`x_3@
;
\begin{scope}[red]
\draw (l)
\zzz0,1`uo...`\tblk{\alpha_1}\beta_7\tblk{\alpha_2}...[inner sep=0pt]-45`y_3@
\zzz-1,1`ul...`\tblk{\alpha_1}\beta_2...180`y_2@
\zzz0.5,1`ulr...`\tblk{\alpha_1}\beta_1...[inner sep=0pt]-45`y_1@
;
\draw (uo)
\zzz1,1`ur...swap`\beta_3\tblk{\alpha_2}...[inner sep=0pt]-45`\tcyan{v'}@
\zzz0.8,1`urr...swap`\beta_6\tblk{\alpha_2}...[inner sep=0pt]-45`y_6@
;
\draw (ur)
\zzz-0.8,1`url...`\beta_4...90`y_4@
;
\draw (ur)
\zzz0,1`ura...swap,pos=0.75`\beta_5...90`y_5@
;
\end{scope}
\draw (ulr)
\zzz-0.5,1`ll...`\alpha_1...90`x_1@
;
\draw (urr)
\zzz0.5,1`lr...swap`\alpha_2...90`x_2@
;
}.
\]
The second step we replace the vertex $v'$ of $T\circ^f_{v,1}U$ with $W$ and denote by $(T\circ^f_{v,1}U)\circ^g_{v',2}W$ the particular component of the composition, where the map $g$ is given above. Then we have
\[
(T\circ^f_{v,1}U)\circ^g_{v',2}W=
\treeo{\oo`-90`x_7@
\draw (o)
\zzz-1.5,1`l,red...`\alpha_3...180`\tred{y_7}@
;
\draw (o)
\zzz1.5,1`r...swap`\alpha_4...[inner sep=0pt]-45`x_6@
\zzz0.5,1`rr...swap`\alpha_6...90`x_4@
;
\draw (r)
\zzz-0.5,1`rl...`\alpha_5...90`x_3@
;
\draw[red] (l)
\zzz0,1`uo...`\tblk{\alpha_1}\beta_7\tblk{\alpha_2}...[inner sep=0pt]-45`y_3@
\zzz-1.5,1`ul...`\tblk{\alpha_1}\beta_2...180`y_2@
\zzz0.5,1`ulr...`\tblk{\alpha_1}\beta_1...0`y_1@
;
\draw[red] (uo)
\zzz1.5,1`ur,cyan...swap`\beta_3\tblk{\alpha_2}...0`\tcyan{z_4}@
;
\draw (ulr)
\zzz-0.5,1`ll...`\tblk{\alpha_1}...90`x_1@
;
\draw[cyan] (ur)
\zzz0,1`vo...swap`\tred{\beta_4\beta_5}\gamma_3\tred{\beta_6\tblk{\alpha_2}}...[inner sep=0pt]-135`z_3@
\zzz-1,1`vl...`\tred{\beta_4}\gamma_1...180`z_1@
;
\draw[cyan] (vo)
\zzz1,1`vr...swap`\gamma_2\tred{\beta_6\tblk{\alpha_2}}...0`z_2@
;
\draw[red] (vl)
\zzz0,1`vla...`\beta_4...90`y_4@
;
\draw[red] (vo)
\zzz0,2`voa...`\beta_5...90`y_5@
;
\draw[red] (vr)
\zzz0,1`vra...swap`\beta_6\tblk{\alpha_2}...135`y_6@
;
\draw (vra)
\zzz0.5,1`vrar...swap`\tblk{\alpha_2}...90`x_2@
;
}.
\]
 Let $\tilde{g}:\hbox{ In}(v')\ra\ver{(W)}$ and $\tilde{f}:\hbox{ In}(v)\ra\ver{(U\circ^{\tilde{g}}_{v'}W)}=\ver{(U)}\sqcup \ver{(W)}\backslash\{v'\}$.
 Choose $$\tilde{g}(\treeoo{
\draw (o)
\zzz0,0.6`l...swap`\beta_4...90`y_4@
;})=z_1, \tilde{g}(\treeoo{
\draw (o)
\zzz0,0.6`l...swap`\beta_5...90`y_5@
;})=z_3, \tilde{g}(\treeoo{
\draw (o)
\zzz0,0.6`l...swap`\beta_6...90`y_6@
;})=z_2\,\text{ and }\,\tilde{f}(\treeoo{
\draw (o)
\zzz0,0.6`l...swap`\alpha_1...90`x_1@
;})=y_1,\tilde{f}(\treeoo{
\draw (o)
\zzz0,0.6`l...swap`\alpha_2...90`x_2@
;})=y_6.$$
Similarly, we also divide into two steps of the right hand side. The first step we replace the vertex $v'$ of $U$ with $W$ and denote by $U\circ^{\tilde{g}}_{v',1}W$ the particular composition. So we have
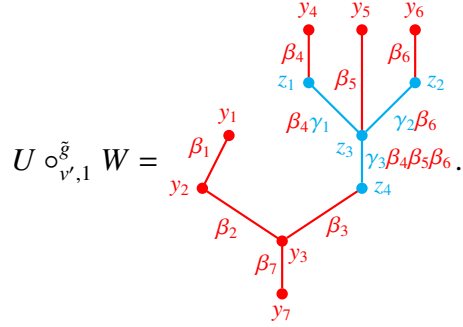
\[
U\circ^{\tilde{g}}_{v',1}W=
\treeo[red]{\oo`-90`y_7@
\draw (o)
\zzz0,1`oo...`\beta_7...[inner sep=0pt]-45`y_3@
\zzz-1.5,1`l...`\beta_2...180`y_2@
\zzz0.5,1`lr...`\beta_1...90`y_1@
;
\draw (oo)
\zzz1.5,1`r,cyan...swap`\beta_3...0`\tcyan{z_4}@
;
\draw[cyan] (r)
\zzz0,1`vo...swap`\gamma_3\tred{\beta_4\beta_5\beta_6}...[inner sep=0pt]-135`z_3@
\zzz-1,1`vl...`\tred{\beta_4}\gamma_1...180`z_1@
;
\draw[cyan] (vo)
\zzz1,1`vr...swap`\gamma_2\tred{\beta_6}...0`z_2@
;
\draw (vl)
\zzz0,1`vla...`\beta_4...90`y_4@
;
\draw (vo)
\zzz0,2`voa...`\beta_5...90`y_5@
;
\draw (vr)
\zzz0,1`vra...`\beta_6...90`y_6@
;
}.
\]
The second step we replace the vertex $v$ of $T$ with $U\circ^{\tilde{g}}_{v',1}W$ and  denote by $T\circ^{\tilde{f}}_{v,2}(U\circ^{\tilde{g}}_{v',1}W)$ the particular component of the composition, where $\tilde{f}$ is given above. Then we have
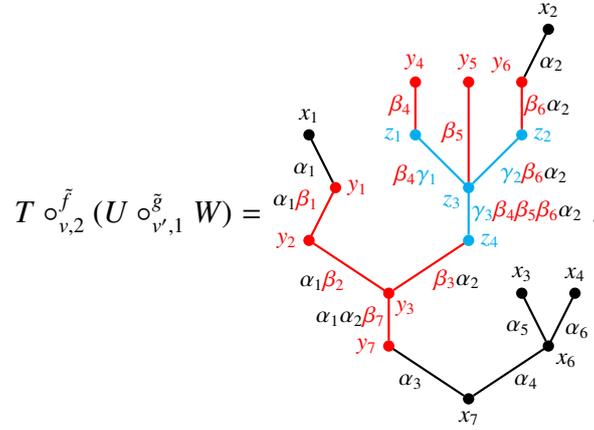
\[
T\circ^{\tilde{f}}_{v,2}(U\circ^{\tilde{g}}_{v',1}W)=
\treeo{\oo`-90`x_7@
\draw (o)
\zzz-1.5,1`l,red...`\alpha_3...180`\tred{y_7}@
;
\draw (o)
\zzz1.5,1`r...swap`\alpha_4...[inner sep=0pt]-45`x_6@
\zzz0.5,1`rr...swap`\alpha_6...90`x_4@
;
\draw (r)
\zzz-0.5,1`rl...`\alpha_5...90`x_3@
;
\draw[red] (l)
\zzz0,1`uo...`\tblk{\alpha_1}\tblk{\alpha_2}\beta_7...[inner sep=0pt]-45`y_3@
\zzz-1.5,1`ul...`\tblk{\alpha_1}\beta_2...180`y_2@
\zzz0.5,1`ulr...`\tblk{\alpha_1}\beta_1...0`y_1@
;
\draw[red] (uo)
\zzz1.5,1`ur,cyan...swap`\beta_3\tblk{\alpha_2}...0`\tcyan{z_4}@
;
\draw (ulr)
\zzz-0.5,1`ll...`\tblk{\alpha_1}...90`x_1@
;
\draw[cyan] (ur)
\zzz0,1`vo...swap`\gamma_3\tred{\beta_4\beta_5\beta_6\tblk{\alpha_2}}...[inner sep=0pt]-135`z_3@
\zzz-1,1`vl...`\tred{\beta_4}\gamma_1...180`z_1@
;
\draw[cyan] (vo)
\zzz1,1`vr...swap`\gamma_2\tred{\beta_6\tblk{\alpha_2}}...0`z_2@
;
\draw[red] (vl)
\zzz0,1`vla...`\beta_4...90`y_4@
;
\draw[red] (vo)
\zzz0,2`voa...`\beta_5...90`y_5@
;
\draw[red] (vr)
\zzz0,1`vra...swap`\beta_6\tblk{\alpha_2}...135`y_6@
;
\draw (vra)
\zzz0.5,1`vrar...swap`\tblk{\alpha_2}...90`x_2@
;
}.
\]
The two subjacent trees of $(T\circ^f_{v,1}U)\circ^g_{v',2}W$ and
$T\circ^{\tilde{f}}_{v,2}(U\circ^{\tilde{g}}_{v',1}W)$ are the same, and the types of all edges as well, due to the commutativity of the semigroup $\Omega$.
\end{exam}

\begin{theorem}
The operad $\pl$ of pre-Lie family algebras is isomorphic to the operad of typed labeled rooted trees $\tr$, via the isomorphism $\Phi:\pl\ra\tr, \rhd_\omega
\mapsto\treeo{
\eoo`b@
\draw (o)
\eee0,1`l`a@
(o)--(l)\xxx swap`\omega@
;}$, where the edge is of type $\omega,\omega\in\Omega.$
\mlabel{thm:iso}
\end{theorem}
\begin{proof}
First, we define an operad morphism $\Phi:\pl\ra\tr.$
Recall that the operad of pre-Lie family algebras is generated by the binary elements $\rhd_\omega,\omega\in\Omega$ with the relations \eqref{relations-operad}. A basis of the vector space $(\ff)_A$ with $|A|=2$, is given by $\{\rhd_\alpha, \alpha\in\Omega\}\sqcup\{\tau\rhd_\alpha, \alpha\in\Omega\}$ where $\tau$ is the nontrivial permutation of two elements. Now set
$$\Phi(\rhd_\alpha)=\treeo{
\eoo`b@
\draw (o)
\eee0,1`l`a@
(o)--(l)\xxx swap`\alpha@
;},\quad \Phi(\tau\rhd_\alpha)=\treeo{
\eoo`a@
\draw (o)
\eee0,1`l`b@
(o)--(l)\xxx swap`\alpha@
;}.$$
Since $\pl=\ff/(R)$, we check that $\Phi(r)=0.$
Hence
\begin{align*}
\Phi(r)&=\Phi\Big(\rhd_\alpha\circ_b\rhd_\beta-\rhd_{\alpha\beta}\circ_a\rhd_\alpha
-\tau_{12}(\rhd_\beta\circ_b\rhd_\alpha-\rhd_{\beta\alpha}\circ_a\rhd_\beta)\Big)\\
&=\treeo{
\eoo`b@
\draw (o)
\eee0,1`l`a@
(o)--(l)\xxx swap`\alpha@
;}\circ_b \treeo{
\eoo`2@
\draw (o)
\eee0,1`l`1@
(o)--(l)\xxx swap`\beta@
;}
-
\treeo{
\eoo`b@
\draw (o)
\eee0,1`l`a@
(o)--(l)\xxx swap`\alpha\beta@
;}\circ_a
\treeo{
\eoo`2@
\draw (o)
\eee0,1`l`1@
(o)--(l)\xxx swap`\alpha@
;}
-\tau_{12}
\left(\treeo{
\eoo`b@
\draw (o)
\eee0,1`l`a@
(o)--(l)\xxx swap`\beta@
;}
\circ_b
\treeo{
\eoo`2@
\draw (o)
\eee0,1`l`1@
(o)--(l)\xxx swap`\alpha@
;}
-
\treeo{
\eoo`b@
\draw (o)
\eee0,1`l`a@
(o)--(l)\xxx swap`\beta\alpha@
;}
\circ_a
\treeo{
\eoo`2@
\draw (o)
\eee0,1`l`1@
(o)--(l)\xxx swap`\beta@
;}\right)\\
&=\left(\treeo{
\eoo`2@
\draw (o)
\eee-1,1`l`a@
(o)--(l)\xxx`\alpha@
;
\draw (o)
\eee1,1`r`1@
(o)--(r)\xxx swap`\beta@
;
}
+
\treeo{
\eoo`2@
\draw (o)
\eee0,1.25`a`1@
(o)--(a)\xxx`\alpha\beta@
;
\draw (a)
\eee0,1.25`aa`a@
(a)--(aa)\xxx`\alpha@
;
}\right)
-
\treeo{
\eoo`b@
\draw (o)
\eee0,1.25`1`2@
(o)--(a)\xxx`\alpha\beta@
;
\draw (a)
\eee0,1.25`aa`1@
(a)--(aa)\xxx`\alpha@
;
}
-
\tau_{12}\left(\treeo{
\eoo`2@
\draw (o)
\eee-1,1`l`1@
(o)--(l)\xxx`\alpha@
;
\draw (o)
\eee1,1`r`a@
(o)--(r)\xxx swap`\beta@
;
}
+\treeo{
\eoo`2@
\draw (o)
\eee0,1.25`a`1@
(o)--(a)\xxx`\beta\alpha@
;
\draw (a)
\eee0,1.25`aa`a@
(a)--(aa)\xxx`\beta@
;
}
-\treeo{
\eoo`b@
\draw (o)
\eee0,1.25`a`2@
(o)--(a)\xxx`\beta\alpha@
;
\draw (a)
\eee0,1.25`aa`1@
(a)--(aa)\xxx`\beta@
;
}\right),\\
&=\treeo{
\eoo`2@
\draw (o)
\eee-1,1`l`a@
(o)--(l)\xxx`\alpha@
;
\draw (o)
\eee1,1`r`1@
(o)--(r)\xxx swap`\beta@
;
}-\treeo{
\eoo`2@
\draw (o)
\eee-1,1`l`a@
(o)--(l)\xxx`\alpha@
;
\draw (o)
\eee1,1`r`1@
(o)--(r)\xxx swap`\beta@
;
}=0.
\end{align*}
So the morphism $\Phi$ is defined on the quotient operad $\pl$.\\

Let us prove that it is bijective. Let $T\in\tm(A)$, we show that it belongs to $\im(\Phi)$ by induction on $|T|$. If $|T|=1$ or $|T|=2$, it is obvious. Let us assume that the result holds for the degree $<|T|$. Let $a$ be the root of typed decorated rooted trees. Up to a permutation, we can write uniquely
$$T=B^+_{a,\omega_1,\ldots,\omega_k}(T_1\cdots T_k)=\treeo{
\eoo`a@
\draw (o)
\eee-1,1`l`T_1@
(o)--(l)\xxx`\omega_1@
;
\draw (o)
\eee1,1`r`T_k@
(o)--(r)\xxx swap`\omega_k@
;
\node at (0,1) {$\dots$};
}$$
where $T_i$ for $1\leq i\leq k,$ is a typed decorated rooted tree of degree strictly less than $|T|$. By the induction hypothesis on $|T|$, $T_i\in\im(\Phi),$ for all $i$, we proceed by induction on $k$. If $k=1$, then
$$T=\treeo{
\eoo`a@
\draw (o)
\eee0,1`l`b@
(o)--(l)\xxx swap`\omega@
;}\circ_b T_1\in\im(\Phi).$$
Let us assume that the result holds for $k-1$, we put
$$T'=B^+_{a,\omega_1,\ldots,\omega_{k-1}}(T_1\cdots T_{k-1})=\treeo{
\eoo`a@
\draw (o)
\eee-1,1`l`T_1@
(o)--(l)\xxx`\omega_1@
;
\draw (o)
\eee1,1`r`T_{k-1}@
(o)--(r)\xxx swap`\omega_{k-1}@
;
\node at (0,1) {$\dots$};
}$$
By the induction hypothesis on $|T|$, $T'\in\im(\Phi).$ Then
$$\treeo{
\eoo`a@
\draw (o)
\eee0,1`l`b@
(o)--(l)\xxx swap`\omega@
;}\circ_a T'=T+T''$$
where $T''$ is a sum of trees with $|T|$ vertices, such that the fertility of the root is $k-1.$ Hence $T''\in \im(\Phi).$ So $T\in\im(\Phi).$\\

Now suppose there is a nontrivial element in the kernel of $\Phi:\pl\to\tr$. That would induce a relation among $\Omega$-typed rooted trees which is not a pre-Lie family relation, which would contradict Theorem~\mref{thm:free}, hence $\Phi$ is an isomorphism.
\ignore{
Let $A$ be a set. Define operad morphism $\phi:\tr\ra\pl$. The morphism $\phi$ implies that the free
$\tr$-algebra generated by $A$, that is to say $\tm(A)$, inherits a pre-Lie family algebra structure defined by
$$x\circ_\omega y:=\gamma\Biggl(\treeo{
\eoo`b@
\draw (o)
\eee0,1`l`a@
(o)--(l)\xxx swap`\omega@
;},x,y\Biggr):=\Biggl(\treeo{
\eoo`b@
\draw (o)
\eee0,1`l`a@
(o)--(l)\xxx swap`\omega@
;}\circ_a x\Biggr)\circ_b y,\,\text{ for}\, x,y\in\tm(A),$$
called the grafting $x$ on $y$. For any trees $T,T'\in\tm(A),$ by definition of the operad composition of $\tr$:
$$T\circ_\omega T'=\sum_{v\in\ver{(T')}}T\rjt{\omega}{v} T',$$
so $\circ_\omega=\rhd_\omega$ for any $\omega$. As $(\tm(A),(\rhd_\omega)_{\omega\in\Omega})$ is the free pre-Lie family algebra generated by $A$ by Theorem~\mref{thm:free}, $\Phi$ is an isomorphism.}
\end{proof}
\ignore{
\begin{coro}
 The operad $\tr$ together with the composition defined by Eq.~(\mref{eq:comp}) defines a pre-Lie family algebra. The pre-Lie family algebra is defined by
 \begin{equation*}
 S\rhd_\omega T:=\Big(\treeo{
\eoo`b@
\draw (o)
\eee0,1`l`a@
(o)--(l)\xxx swap`\omega@
;}\circ_a S\Big)\circ_b T.
 \end{equation*}
\end{coro}

\begin{proof}
This result can be directly obtained Theorem~\mref{thm:iso}.
\end{proof}
}
\section{Zinbiel and Pre-Poisson family algebras}\label{sec:pp}
In this section, we mainly generalize Aguiar's results~\cite{A} , that is, the relationships between pre-Lie family algebras, Zinbiel family algebras and pre-Poisson family algebras. Zinbiel algebras were introduced by Loday~\cite{L1}, see also~\cite{Liv}. We propose here the following definition of a left Zinbiel family algebra.
\begin{defn}
Let $\Omega$ be a commutative semigroup.
A left zinbiel family algebra is a vector space $A$ together with binary operations $\ast_\omega:A\times A\ra A$ for $\omega\in\Omega$ such that
\begin{equation*}
  x\ast_\alpha(y\ast_\beta z)=(x\ast_\alpha y)\ast_{\alpha\beta} z+(y\ast_\beta x)\ast_{\alpha\beta} z,
\end{equation*}
where $x,y,z\in A$ and $\alpha,\beta\in\Omega.$
\end{defn}

\begin{prop}
Let $\Omega$ be a commutative semigroup. Let $(A,(\prec_\omega,\succ_\omega)_{\omega\in\Omega})$ be a commutative dendriform family algebra, i.e. a dendriform family algebra satisfying
$$x\succ_\omega y=y\prec_\omega x.$$
Define the zinbiel family product $$x\ast_\omega y:=x\succ_\omega y=y\prec_\omega x.$$
Then $(A,(\ast_\omega)_{\omega\in\Omega})$ is a zinbiel family algebra.
\mlabel{prop:zin}
\end{prop}
\begin{proof}
Since $(A,(\prec_\omega,\succ_\omega)_{\omega\in\Omega})$ is a dendriform family algebra,
we have
$$x\succ_\alpha(y\succ_\beta z)=(x\succ_\alpha y +y\succ_\beta x)\succ_{\alpha\beta}z.$$
Hence
$$x\ast_\alpha(y\ast_\beta z)=(x\ast_\alpha y +y\ast_\beta x)\ast_{\alpha\beta}z,$$
as required.
\end{proof}

\begin{defn}
A Poisson algebra is a triple $(A,\cdot,\{,\})$, where $(A,\cdot)$ is a commutative algebra, $(A,\{,\})$ is a Lie algebra, and the following condition holds:
\begin{equation}
\{x,y\cdot z\}=\{x,y\}\cdot z+y\cdot\{x,z\}.
\mlabel{eq:po}
\end{equation}
\end{defn}

Combining zinbiel family algebra and pre-Lie family algebra,
we propose the following definition.

\begin{defn}
A left pre-Possion family algebra is a triple $(A,(\ast_\omega,\rhd_\omega)_{\omega\in\Omega})$ where $(A,(\ast_\omega)_{\omega\in\Omega})$ is a left zinbiel family algebra and $(A,(\rhd_\omega)_{\omega\in\Omega})$ is a left pre-Lie family algebra. The following conditions hold:
\begin{align}
(x\rhd_\alpha y-y\rhd_\beta x)\ast_{\alpha\beta}z&=x\rhd_\alpha(y\ast_\beta z)-y\ast_\beta(x\rhd_\alpha z),\mlabel{eq:pp1}\\
(x\ast_\alpha y+y\ast_\beta x)\rhd_{\alpha\beta}z&=x\ast_\alpha(y\rhd_\beta z)+y\ast_\beta(x\rhd_\alpha z),\mlabel{eq:pp2}
\end{align}
where $x,y,z\in A$ and $\alpha,\beta\in\Omega.$
\mlabel{defn:prp}
\end{defn}
%

\begin{prop}
Let $\Omega$ be a commutative semigroup.
\begin{enumerate}
\item Let $(A,\{,\, \},(P_\omega)_{\omega\in\Omega})$ be a Rota-Baxter family Lie algebra of weight $0$. Define new operations on $A$ by $$x\rhd_\omega y=\{P_\omega(x),y\}.$$ Then $(A,(\rhd_\omega)_{\omega\in\Omega})$ is a left pre-Lie family algebra.\mlabel{it:pl}
\item Let $(A,\cdot,(P_\omega)_{\omega\in\Omega})$ be a commutative Rota-Baxter family algebra of weight $0$. Define new operations on $A$ by $$x\ast_\omega y=P_\omega(x)\cdot y.$$
     Then $(A,(\ast_\omega)_{\omega\in\Omega})$ is a left zinbiel family  algebra.\mlabel{it:p2}
    \end{enumerate}
    \mlabel{prop:zp}
\end{prop}
\begin{proof}
\mref{it:pl}. Since $(A,\{,\, \},(P_\omega)_{\omega\in\Omega})$ is a Rota-Baxter family Lie algebra of weight $0$, we have
\begin{align*}
&\{P_\alpha(x),\{P_\beta(y), z\}\}+\{P_\beta(y),\{z,P_\alpha(x)\}\}+\{z,\{P_\alpha(x),P_\beta(y)\}\}\\
={}&\{P_\alpha(x),\{P_\beta(y),z\}\}+\{P_\beta(y),\{z,P_\alpha(x)\}\}
  +\{z,P_{\alpha\beta}\{P_\alpha(x),y\}\}+\{z,P_{\alpha\beta}\{x,P_\beta(y)\}\}=0
\end{align*}
We get
\begin{equation}
\{P_\alpha(x),\{P_\beta(y),z\}\}=-\{P_\beta(y),\{z,P_\alpha(x)\}\}
-\{z,P_{\alpha\beta}\{P_\alpha(x),y\}\}-\{z,P_{\alpha\beta}\{x,P_\beta(y)\}\}.
\mlabel{eq:pp}
\end{equation}
Then
\begin{align*}
&x\rhd_\alpha(y\rhd_\beta z)-(x\rhd_\alpha y)\rhd_{\alpha\beta}z\\
={}&x\rhd_\alpha\{P_\beta(y),z\}-\{P_\alpha(x),y\}\rhd_{\alpha\beta}z\\
={}&\{P_\alpha(x),\{P_\beta(y),z\}\}-\{P_{\alpha\beta}\{P_\alpha(x),y\},z\}\\
={}&{-}\{P_\beta(y),\{z,P_\alpha(x)\}\}-\{z,P_{\alpha\beta}\{P_\alpha(x),y\}\}
-\{z,P_{\alpha\beta}\{x,P_\beta(y)\}\}-\{P_{\alpha\beta}\{P_\alpha(x),y\},z\}\\
&\hspace{3cm}(\text{by Eq.~(\mref{eq:pp})})\\
={}&{-}\{P_\beta(y),\{z,P_\alpha(x)\}\}-\{z,P_{\alpha\beta}\{x,P_\beta(y)\}\}\\
={}&\{P_\beta(y),\{P_\alpha(x),z\}\}-\{P_{\beta\alpha}\{P_\beta(y),x\},z\}\quad(\text{ by $\Omega$ being a commutative semigroup})\\
={}&y\rhd_\beta(x\rhd_\alpha z)-(y\rhd_\beta x)\rhd_{\beta\alpha}z.
\end{align*}
\mref{it:p2}. Since $(A,\cdot,(P_\omega)_{\omega\in\Omega})$ is a commutative Rota-Baxter family algebra, we have
\begin{align*}
x\ast_\alpha(y\ast_\beta z)&=x\ast_\alpha\Big(P_\beta(y)\cdot z\Big)\\
&=P_\alpha(x)\cdot\Big(P_\beta(y)\cdot z\Big)
=\Big(P_\alpha(x)\cdot P_\beta(y)\Big)\cdot z\\
&=P_{\alpha\beta}\Big(P_\alpha(x)\cdot y+x\cdot P_\beta(y)\Big)\cdot z\\
&=P_{\alpha\beta}\Big(P_\alpha(x)\cdot y\Big)\cdot z+P_{\alpha\beta}\Big(P_\beta(y)\cdot x\Big)\cdot z\\
&=(x\ast_\alpha y)\ast_{\alpha\beta}z+(y\ast_\beta x)\ast_{\alpha\beta} z.
\end{align*}
This completes the proof.
\end{proof}

In view of the the above results, one expects that a Rota-Baxter family $(P_\omega)_{\omega\in\Omega}$ on a Poisson algebra will alow us to construct a pre-Poisson family algebra structure on it.
\begin{prop}
Let $\Omega$ be a commutative semigroup. Let $(A,\cdot,\{,\})$ be a Poisson algebra and let $P_\omega:A\ra A$ be a Rota-Baxter family of weight $0$. Define new operations on $A$ by
\begin{equation}
x\ast_\omega y=P_\omega(x)\cdot y\,\text{ and }\, x\rhd_\omega y=\{P_\omega(x), y\},\,\text{ for }\, \omega\in\Omega.
\mlabel{eq:ast}
\end{equation}
Then $(A,(\ast_\omega,\rhd_\omega)_{\omega\in\Omega})$ is a left pre-Poisson family algebra.
\end{prop}
\mlabel{prop:poi}
\begin{proof}
We first prove Eq.~(\mref{eq:pp1}), then
\begin{align*}
(x\rhd_\alpha y-y\rhd_\beta x)\ast_{\alpha\beta}z&=\Big(\{P_\alpha(x),y\}-\{P_\beta(y),x\}\Big)\ast_{\alpha\beta}z\\
&=\{P_\alpha(x),y\}\ast_{\alpha\beta}z-\{P_\beta(y),x\}\ast_{\alpha\beta}z
\quad(\text{by Eq.~(\mref{eq:ast})})\\
&=P_{\alpha\beta}\{P_\alpha(x),y\}\cdot z-P_{\alpha\beta}\{P_\beta(y),x\}\cdot z\quad(\text{by Eq.~(\mref{eq:ast})})\\
&=P_{\alpha\beta}\{P_\alpha(x),y\}\cdot z+P_{\alpha\beta}\{x,P_\beta(y)\}\cdot z\\
&=\{P_\alpha(x),P_\beta(y)\}\cdot z+P_\beta(y)\cdot\{P_\alpha(x),z\}-P_\beta(y)\cdot\{P_\alpha(x),z\}\\
&=\{P_\alpha(x),P_\beta(y)\cdot z\}-P_\beta(y)\cdot\{P_\alpha(x),z\}\quad(\text{by Eq.~(\mref{eq:po})})\\
&=x\rhd_\alpha(y\ast_\beta z)-y\ast_\beta(x\rhd_\alpha z).
\end{align*}
Second, we prove Eq.~(\mref{eq:pp2}), we have
\begin{align*}
(x\ast_\alpha y+y\ast_\beta x)\rhd_{\alpha\beta} z&=\Big(P_\alpha(x)\cdot y+P_\beta(y)\cdot x\Big)\rhd_{\alpha\beta}z\quad(\text{by Eq.~(\mref{eq:ast})})\\
&=\Big\{P_{\alpha\beta}\Big(P_\alpha(x)\cdot y+P_\beta(y)\cdot x\Big), z\Big\}\quad(\text{by Eq.~(\mref{eq:ast})})\\
&=\{P_\alpha(x)\cdot P_\beta(y), z\}\\
&=-\{z,P_\alpha(x)\cdot P_\beta(y)\}\\
&=-\Big(\{z,P_\alpha(x)\}\cdot P_\beta(y)+P_\alpha(x)\cdot\{z,P_\beta(y)\}\Big)\quad(\text{by Eq.~(\mref{eq:po})})\\
&=P_\alpha(x)\cdot\{P_\beta(y),z\}+P_\beta(y)\cdot\{P_\alpha(x), z\}\\
&=x\ast_\alpha\{P_\beta(y),z\}+y\ast_\beta\{P_\alpha(x),z\}\quad(\text{by Eq.~(\mref{eq:ast})})\\
&=x\ast_\alpha(y\rhd_\beta z)+y\ast_\beta(x\rhd_\alpha z).
\end{align*}
This completes the proof.
\end{proof}
\noindent {\bf Acknowledgments}:
This work was supported by the National Natural Science Foundation
of China (Grant No.\@ 11771191 and 11861051).

\end{document}